\pdfoutput=1
\documentclass[10pt]{article}

\usepackage[hidelinks,colorlinks = true, citecolor = {teal}, breaklinks, backref=section, pagebackref=true, linktoc=all]{hyperref}

\usepackage{xcolor}

\usepackage[]{darkmode}	
\usepackage[numbers]{natbib} 

\usepackage[linesnumbered,lined,commentsnumbered,noend]{algorithm2e}
\SetKwComment{Comment}{\qquad/* }{ */}
\RestyleAlgo{ruled}
\DontPrintSemicolon
\SetAlgoSkip{}
\SetKwInOut{Input}{Input\hspace{-0.5em}}

\let\oldnl\nl
\newcommand{\nlnonumber}{\renewcommand{\nl}{\let\nl\oldnl}}
\SetAlCapFnt{\footnotesize}

\usepackage[utf8]{inputenc}
\usepackage[T1]{fontenc}
\usepackage{enumitem} 	
\usepackage{xkcdcolors}

\usepackage{geometry}
\usepackage{setspace}
\usepackage{parskip} 
\usepackage{graphicx}     
\usepackage{tcolorbox}    
\tcbuselibrary{breakable}
\usepackage{subcaption} 
\usepackage{wrapfig}	
\usepackage[font=footnotesize]{caption}
\usepackage{booktabs}	
\usepackage{multirow}
\usepackage{multicol}
\usepackage{orcidlink}	
\usepackage[color = xkcdWheat, bordercolor = xkcdWheat, textsize=tiny, textwidth=14mm]{todonotes}
\usepackage{soul}	
\setstcolor{xkcdRed}

\usepackage{amsfonts}
\usepackage{amsmath}
\usepackage{amssymb,amsthm}
\usepackage{mathtools} 
\usepackage{commath} 
\usepackage{xfrac} 
\usepackage{stmaryrd} 
\usepackage{dsfont}  
\usepackage{siunitx} 
\usepackage{arydshln} 
\usepackage{accents}	
\usepackage[toc,page]{appendix}
\AtBeginEnvironment{appendices}{\crefalias{section}{appendix}}

\usepackage[capitalise]{cleveref}	
\crefformat{equation}{#2(#1)#3}
\crefrangeformat{equation}{#3#1#4 to #5#2#6}
\crefmultiformat{equation}{#2(#1)#3}{ and~#2(#1)#3}{, #2(#1)#3}{, and~#2(#1)#3}

\theoremstyle{definition}
\newtheorem{definition}{Definition}[section]

\theoremstyle{remark}
\newtheorem{remark}{Remark}

\theoremstyle{definition}
\newtheorem{theorem}{Theorem}[section]
\newtheorem{proposition}[theorem]{Proposition}
\newtheorem{lemma}[theorem]{Lemma}
\newtheorem{corollary}{Corollary}[theorem]

\theoremstyle{definition}

\newtheorem{assumptions}{Assumptions}[section]

\makeatletter
\renewenvironment{proof}[1][\proofname]{\par
  \pushQED{\qed}%
  \normalfont \topsep0pt\relax 
  \trivlist
  \item[\hskip\labelsep
        \itshape
    #1\@addpunct{.}]\ignorespaces
}{%
  \popQED\endtrivlist\@endpefalse
}
\makeatother

\newtcolorbox{highlighter}{colback=xkcdGoldenrod, boxrule=0pt, sharp corners, boxsep=0pt, left=\fboxsep, right=\fboxsep, parbox=false, breakable}

\newcommand{\embh}[1]{%
\IfDarkModeTF{{\color{white!80!\thepagecolor}\textbf{#1}}}{{\color{black!80}\textbf{#1}}}%
}

\definecolor{algcolor}{rgb}{0.97,0.97,0.97} 
\definecolor{IMAcolor}{RGB}{187,39,113} 
\hypersetup{linkcolor = IMAcolor}
\colorlet{TocColor}{white!10!black}

\newcommand{\R}{\mathbb{R}}		
\newcommand{\C}{\mathbb{C}}		
\newcommand{\N}{\mathbb{N}}		

\newcommand{\Convh}{\mathbb{M}}	            

\newcommand{\llb}{\llbracket}
\newcommand{\rrb}{\rrbracket}
\newcommand{\bfu}{\mathbf{u}}
\newcommand{\bfv}{\mathbf{v}}

\newcommand{\rhob}{\pmb{\rho}}

\DeclareMathOperator{\diag}{diag}
\DeclareMathOperator{\dive}{div}

\newcommand{\Ones}[1][n]{\mathds{1}_{#1}}		
\newcommand{\into}{\int\limits_\Omega}			
\newcommand{\intdo}{\iint\limits_{\Omega\times\Omega}}		

\allowdisplaybreaks			
\usepackage{stackengine,scalerel}
\usepackage{dashbox}

\newcommand\dsquare{\ThisStyle{\ensurestackMath{%
  \stackinset{c}{}{c}{}{\scalebox{.45}{\color{\thepagecolor}$\SavedStyle\blacksquare$}}
  {\SavedStyle\blacksquare}}}}
  
\newcommand{\smallblacksquare}{\raisebox{.08\height}{$\scaleobj{0.5}{\blacksquare}$}}

\newcommand{\disc}{\raisebox{.05em}{$\bigcirc$}}
\newcommand{\smalldisc}{\scaleobj{0.9}{\newmoon}}
\newcommand\ddisc{\ThisStyle{\ensurestackMath{%
  \stackinset{c}{}{c}{}{\scalebox{.54}{\color{\thepagecolor}$\SavedStyle\newmoon$}}
  {\SavedStyle\scalebox{1.25}{\newmoon}}}}}

\usepackage[nointegrals]{wasysym} 
\usepackage{fontawesome} 

\title{Singularities in phase separation models: 
\\a spectral element  approach \\ for
the nonlocal Cahn--Hilliard equation}
\author{%
    Andrés Miniguano-Trujillo%
    \thanks{Maxwell Institute for Mathematical Sciences, The University of Edinburgh and Heriot-Watt University, Bayes Centre, Edinburgh, United Kingdom}
    \thanks{Department of Neuroimaging, King's College London, London, United Kingdom 
    (\texttt{Andres.Miniguano-Trujillo@kcl.ac.uk} \orcidlink{0000-0002-0877-628X})}
    \\
    Andrea Poiatti 
    \thanks{Dipartimento di Scienze Matematiche, Fisiche e Informatiche, Università di Parma, Parma, Italy
    (\texttt{andrea.poiatti@unipr.it} \orcidlink{0009-0000-2217-742X})}
        \\
    Maurizio Grasselli
    \thanks{Dipartimento di Matematica, Politecnico di Milano, Milano, Italy
    (\texttt{maurizio.grasselli@polimi.it} \orcidlink{0000-0003-2521-2926})}
    \\
    Benjamin D. Goddard\thanks{School of Mathematics and Maxwell Institute for Mathematical Sciences, The University of Edinburgh, Edinburgh, United Kingdom 
    (\texttt{b.goddard@ed.ac.uk} \orcidlink{0000-0002-8781-014X}, \texttt{j.pearson@ed.ac.uk} \orcidlink{0000-0002-6063-1766} )}
    \\
    John W. Pearson\footnotemark[4]
    }

\date{}

\begin{document}
\spacing{1.213}
\maketitle

\vspace{-1\baselineskip}

\begin{abstract}

    The nonlocal Cahn--Hilliard equation provides a natural extension of the classical model for phase separation by incorporating long-range interactions through a singular convolution kernel. While this formulation admits a rich existence and regularity theory, its numerical approximation remains challenging: discretisation of the nonlocal term leads to dense operators, and the singularity of the kernel requires special treatment in collocation-based schemes. In this work, we develop an efficient and error-controlled numerical framework for the nonlocal Cahn--Hilliard system with constant mobility, logarithmic potential, Newtonian interaction kernel, and no-flux boundary conditions. Our approach is based on a pseudospectral multishape method that accurately approximates the action of singular convolution operators. We present high-resolution numerical solutions for this nonlocal system of equations that can be achieved with limited computational resources.

\end{abstract}


%
%

\section{Introduction}\label{ch:NLCH_Intro}

The primary goal of this study is to develop an efficient numerical scheme to approximate solutions of the nonlocal Cahn--Hilliard system. In particular, we will focus on its variant with constant mobility, singular potentials, and two-component mixtures. Notwithstanding, our techniques will be general enough to be applicable to more wide-ranging systems. 

The Cahn--Hilliard system was proposed in \cite{Cahn1958} as a macroscopic model for the formation and evolution of microstructures during the phase separation of binary alloys. This process consists of an early stage where the spinodal decomposition takes place, and a coarsening process follows it. As a result, this process leads to the segregation of the system into spatial subdomains, where one of the constituents of the mixture prevails. 

In this work, we study the nonlocal Cahn--Hilliard (CH) system derived in \cite{Giacomin1996}, which describes the phase separation of a binary mixture in periodic settings. It can be understood as an instance of a generalised dynamical density functional theory (DDFT) system with no-flux boundary conditions. Specifically, we consider\footnote{The differentiability of \(\mathcal{F}\) should be understood in terms of the Gâteaux derivative and its Riesz representative; i.e., the gradient functional. 
}
\begin{equation}\label{eq:GNNFT_CH}
	\od{\rho}{t} = \nabla \cdot \del{ m(\rho) \nabla \frac{ \delta \mathcal{F}[\rho] }{ \delta \rho }  },
\end{equation}
where \(m\) is a mobility term, often constant or bounded but always non-negative, and the free energy \(\mathcal{F}\) is the difference between the \embh{internal} energy \( \mathcal{F}_{\text{int}} \) and an entropy term \(\mathcal{S}\):
\begin{align}
	\label{free}
	\mathcal{F}[\rho] 	&\coloneqq \mathcal{F}_{\text{int}}[\rho] - \mathcal{S}[\rho],
	\qquad
	\mathcal{F}_{\text{int}}[\rho]  \coloneqq  -\frac{1}{2} \intdo K(x-y) \rho(x) \rho(y) \dif x \dif y,
	\qquad
	\mathcal{S}[\rho] \coloneqq  -\into F(\rho) \dif x  ,
\end{align}
and also
\begin{equation}\label{eq:LogPotential}
	F: [-1,1] \ni s  \longmapsto 
	\frac{\theta}{2} \sbr{ (1+s) \log (1+s) + (1-s) \log (1-s)  } \in [ 0, \theta \, {\log 2} ],
\end{equation}
 where \(\theta >0\) is the temperature of the system. The free \embh{energy density}, also known as the \embh{logarithmic potential}, \(F\) is used to distinguish the pure states \( \pm 1\)\footnote{Here, we present the logarithmic potential as our baseline choice for the nonlocal CH system. However, we include a general existence framework that covers a general class of potentials such as the double-well potential \(F(s) = (s-1)^2\).}. Moreover,
\begin{equation}
\label{eq:DerivativesOfF}
	F': (-1,1) \ni s  \longmapsto  \frac{\theta}{2} \log \frac{1+s}{1-s} \in \R
	\qquad\text{and}\qquad
	F'': (-1,1) \ni s  \longmapsto   \frac{\theta}{1-s^2} \in [\theta,\infty).
\end{equation}
We delay for now the properties required by \(K\). 
%
%
Note that the nonlocal Cahn-Hilliard equation can be also seen as a suitable approximation of the classical local Cahn-Hilliard equation if one replaces the gradient square term in the free energy functional with its nonlocal counterpart (cf., e.g.,  \cite{Gal2017}, see also \cite{Scarpa2,Scarpa1} for a rigorous analysis of the relation between the classical local model and the nonlocal one).

 Let $\Omega\subset \R^d$ be a bounded domain with Lipschitz boundary, and set, for $T>0$, \( Q \coloneqq \Omega \times (0,T)\).  Evaluating the terms of \cref{eq:GNNFT_CH}, the general form of the system reads as follows:
\begin{equation}
	\label{PDS:NL-CH}
		\left\{
		\begin{aligned}
			\dod{\rho}{t} &= \dive \big( m(\rho) \nabla \mu \big)	\qquad	&&\text{in } Q,
			\\
			\mu &= F'(\rho) - K \star \rho							&&\text{in } Q,
			\\
			\partial_n \mu &= 0									&&\text{on } \partial\Omega \times (0,T),
			\\
			\rho(0) &= \rho_0									&&\text{in } \Omega.
		\end{aligned}
		\right.
\end{equation}
This partial differential equation system will be addressed as the \embh{nonlocal Cahn--Hilliard system}%
%
%
\footnote{
An alternative form is to fully expand the argument of the divergence form, which yields
\begin{equation}\label{eq:Expanded_CH}
	\od{\rho}{t} = \nabla \cdot \del{ m(\rho) F''(\rho) \nabla \rho  - m(\rho) \nabla K \star \rho  },
\end{equation}
see, e.g.,  \cite{Cavaterra2025a}  and references therein.
Also known as the \embh{nonlocal Cahn--Hilliard equation}, 
the right-hand side in \cref{eq:Expanded_CH} is a nonlinear diffusion term provided that \(mF''\) is strictly positive on \( (-1,1) \). The spatial convolution with a sufficiently smooth and fast-decaying kernel \(K\) models nonlocal aggregation. 
}.

Determining numerical solutions of \cref{PDS:NL-CH} leads to significant challenges. Specifically, the nonlocal nature of the convolution term \(K\), the nonlocal differential algebraic equation in \(\mu\), and the singularity of \(F'\) for pure states cause difficulties when employing standard numerical discretisation schemes.

The core contributions of this paper are summarised below:
\begin{itemize}[leftmargin=1.2em, itemsep=0.0em]
    \item We extend the mathematical analysis of the nonlocal Cahn--Hilliard model with singular potential $F$ to Lipschitz domains satisfying the cone condition and, exploiting the recent results on the strict separation property \cite{P}, we prove a convergence rate of each weak solution to the unique equilibrium also in three dimensions, which was not yet obtained in the literature.

    \item We develop a spectral element method to approximate the convolution against a singular kernel, which appears to be new in the literature. The approach is designed to be flexible to different kernels and domains.

    %
    \item We provide, for the nonlocal Cahn-Hilliard equation with a singular kernel $K$ and singular potential $F$, a detailed analytical study of the approximation scheme, establishing error estimates which are also computationally verified.

    \item We solve the nonlocal Cahn--Hilliard system and present extensive numerical tests, including a study of different parametric settings and examples.

    \item We provide open-source code for the above tests, available at  
\begin{center}
	\href{https://github.com/DDFT-Modelling/NLCH}{\texttt{https://github.com/DDFT-Modelling/NLCH}}
\end{center}
\end{itemize}

The rest of this paper is structured as follows. In \cref{ch:CH_WP_and_Setup}, we provide the existence and uniqueness theory for the nonlocal Cahn--Hilliard system \cref{PDS:NL-CH}  with constant mobility, together with some further refinement concerning the longtime behavior of trajectories. 
In \cref{ch:Spectral_Element_Newtonian_Potential}, we develop a spectral element method to numerically approximate the convolution operator of a kernel which is singular at the origin, yet integrable, through a decomposition technique.
We validate the approximation through several numerical experiments that showcase the effectiveness of the method. 
Later in \cref{ch:NLCH_Numerical_Experiments} we study several configurations of the system \cref{PDS:NL-CH} 
including scalings of the kernel, a set of initial conditions, a regularised potential for long-time approximations, and a linear combination of kernels.  
Finally, in \cref{ch:NLCH_Conclusions}, we present our conclusions.

\section{Existence theory and framework}\label{ch:CH_WP_and_Setup}

In this section, we discuss the background on the well-posedness of the nonlocal Cahn--Hilliard system \cref{PDS:NL-CH} for $d$-dimensional bounded domains, $d \in \{2,3\}$. In particular, we determine conditions for the convolution kernel \(K\), the potential \(F\), and the domain \(\Omega\) that yield a unique solution to the system. We then discuss some existing results about the well-posedness of strong solutions, together with the validity of the instantaneous strict separation property; i.e., the solution staying uniformly away from the pure phases $\pm 1$ from any positive time onwards. Then we give some results concerning the analysis of the long-time behavior of the solution, especially concerning the convergence to a unique equilibrium. Finally, we describe our setup; i.e., the choice of \( (K,F,\Omega) \) that will guide our numerical approximation methods.

\subsection{Well-posedness}\label{sec:CH_WP}

The existence and uniqueness of weak solutions of problem \cref{PDS:NL-CH} has been studied in the context of more general differential systems. 
A fixed point argument was used in \cite[Theorem 3.5]{Gajewski2003} for a particular choice of \(F\) and the mobility. Existence on the torus is studied in \cite[Theorem 4.1]{Giacomin1998} for globally continuously differentiable \(K\) and smooth \(F\). 
In \cite{Gal2017}, a general approach with some assumptions on the singular potential \(F\) and \(K\) is presented for the constant mobility case.
%

In what follows, we state some existence and uniqueness results mainly from \cite{Gal2017}. We discuss an extension of the results that follow from minor changes in the original proofs and where these have to be applied. 

\begin{assumptions}\label{assu:NLCH}
We focus our attention on an interaction kernel \(K\) and a singular potential \(F\) that satisfy
	\begin{enumerate}[label=(H.\arabic*),  itemsep=0.25em]
		\item \( K\in W_{\text{loc}}^{1,1} (\R^d) \) with \( K(x) = K(-x)\), for almost any $x\in\R^d$, \(d \in \{2,3\}\).
		
		\item \(F \in C[-1,1] \cap C^2(-1,1)\) such that
		\(
			\lim\limits_{s\to -1} F'(s) = -\infty
		\),
		\(
			\lim\limits_{s\to 1} F'(s) = \infty,
		\)
		and
		\(
			F''(s) \geq c_0 > 0.
		\)
	\end{enumerate}
\end{assumptions}

Let us further introduce the following notations and spaces: Let \( V = H^1(\Omega)\), and define the \embh{average value} of a function \( \overline{u} = \frac{1}{\mathcal{L}^d(\Omega)} \int\limits_\Omega u \dif x\), where $\mathcal L^d$ is the $d$-dimensional Lebesgue measure.


We are particularly interested in extending the Cahn--Hilliard theory to domains with corners as these naturally arise in the discretisation of differential equations, where computational meshes often involve polygonal or polyhedral elements. 

\begin{definition}[{\citealp[\S 1.1.9, Definition 2]{Mazya2011}}]
    A domain \(\Omega \subsetneq \R^d\) possesses the \embh{cone property} if each point of \(\Omega\) is the vertex of a cone contained in \(\Omega\) along with its closure, the cone being represented by the inequalities \( \sum\limits_{k=1}^{d-1} x_k^2 < b x_d^2 \) and \( x_d \in (0,a)\) in some Cartesian system with \( a,b\) being two positive constants\footnote{Notice that the constants are independent of the chosen vertex. This is also known as the uniform cone property.}.
\end{definition}

It can be proven that any bounded domain has the cone property if and only if it is a Lipschitz domain \cite[Theorem 1.2.2]{Grisvard2011a}. As a consequence, any bounded, open, and convex set is a Lipschitz domain.

Let us now state the following extension of the Poincaré--Wirtinger (also known as Poincaré--Steklov) inequality for Lipschitz domains:
\begin{proposition}[{\citealp[\S 1.1.11, p.\,20]{Mazya2011}, \citealp[Theorem 4.6.1]{Ziemer1989a}, \citealp[Part I, Lemma 3.24]{Ern2021a}}]\label{prop:PW}
	Let \(\Omega \subset \R^d\) be a bounded Lipschitz domain satisfying the cone condition and let \( p \in [1,\infty]\). Then there exists a constant \(C_{P}(p,\Omega)>0\) such that for every \( u \in W^{1,p}(\Omega)\) it holds:
	\[
		\| u - \overline{u} \|_{L^p(\Omega)} \leq C_P(p,\Omega) \| \nabla u \|_{ L^p(\Omega) }.
	\]
\end{proposition}

If \(\Omega\) is a hypercube with largest side-length \(\ell>0\) and \(u\) is continuously differentiable, it can be proven that the Poincaré constant in \cref{prop:PW} is \( C_{P}(p
,\Omega) = d \ell \) \cite[Theorem 3.1.2]{Krylov2023a}.


Additionally, we can present the following Gagliardo--Nirenberg inequality, which applies to bounded Lipschitz domains: 
\begin{proposition}\label{prop:GN}
	Let \(\Omega \subset \R^d\), with $d\in\{2,3\}$, be a bounded Lipschitz domain, and let \(u \in H^1(\Omega) \). Define the admissible range of $p$ as \( [2,\infty) \) for \(d = 2\), and \( [2,6] \) for \( d= 3\).
    Then there exists a constant $C_{p,\Omega} > 0$ such that:
    \[
        \|u\|_{L^p(\Omega)} 
        \leq 
        C_{p,\Omega} \|u\|_{L^2(\Omega)}^\alpha \|u\|_{H^1(\Omega)}^{1 - \alpha} ,
    \]
    where the exponent $\alpha$ and the constant $C_{p,\Omega}$ satisfy
    \begin{itemize}
        \item for $d = 2$: $\alpha = \sfrac{2}{p}$ and $C_{p,\Omega} = C_\Omega \sqrt{p}$, where $C_\Omega > 0$ is independent of $p$;
        
        \item for $d = 3$: $\alpha = \sfrac{3}{p} - \sfrac{1}{2}$, and $C_{p,\Omega} > 0$ may depend on both $p$ and $\Omega$.
    \end{itemize}
\end{proposition}
\begin{proof}
    Since $\Omega$ is a bounded Lipschitz domain, there exists a suitable function extension operator from $\Omega$ to $\R^d$; see for instance \cite[Proposition 2.3]{BM}. Therefore the results follow from the classical Gagliardo--Nirenberg's inequalities applied to $d\in\{2,3\}$; see, for instance, \cite[Theorem 2.44]{GN}. For a thorough exposition on the topic, see \cite{BM}.
\end{proof}

The existing proof of well-posedness of weak solutions to \cref{PDS:NL-CH} in \cite{Gal2017} under \cref{assu:NLCH} also assumes that \(m = 1\) and \(\Omega\) is at least of class \(C^1\). We can obviously generalise to any other constant \(m>0\) in \cref{PDS:NL-CH} through an appropriate scaling in time. As we need to consider a slightly less regular bounded domain \(\Omega\) for constructing numerical approximations, we can review the proof and extend it, for instance, to polygonal domains, which constitute a particular class of Lipschitz domains. In particular, \cref{prop:PW} replaces inequality (2.2) in \cite{Gal2017} and any other occurrences pertaining to the existence and uniqueness proofs that follow. Therefore, the existence and uniqueness of global-in-time weak solutions to \cref{PDS:NL-CH} are ensured in this case as well.

%
Additionally, the well-posedness of strong solutions is established analogously to \cite[Theorem 4.1]{AGGP} in the 2D case and \cite[Theorem 2.2]{PS} in 3D, both under zero advective velocity. Indeed, this proof is based on the compact Sobolev embedding $H^1(\Omega)\hookrightarrow L^2(\Omega)$, which clearly also holds in the case of bounded Lipschitz domains, as well as on \cref{prop:PW}. In the end, this also allows us to prove that any weak solution to the nonlocal Cahn--Hilliard equation instantaneously regularises initial data. We summarise all of these results next in \cref{known}. We refer the reader to \cite[Remark 3.4]{P} for some more precise references on the proof. All the results are here presented for the case of constant mobility $m=1$ for simplicity, but this can be easily generalized to general $m>0$ after suitable rescaling.

\begin{theorem}
	\label{known}
	Let \(\Omega \subset \R^d\), with $d\in\{2,3\}$, be a bounded Lipschitz domain satisfying the cone condition. Assume that \cref{assu:NLCH} hold, and also that $\rho_0\in L^\infty(\Omega)$ such that $\norm{\rho_0}_{L^\infty} \leq 1$ and $\norm{\overline{\rho}_0} < 1$. Assume also that the mobility $m=1$. Then, for any $T>0$, there exists a unique weak solution to \cref{PDS:NL-CH}, such that:
    \begin{subequations}
    \begin{align*}
        &\rho \in L^\infty \del{\Omega \times (0,T)}		\qquad \quad \text{ with } |\rho| < 1 \text{ almost everywhere in } \Omega \times (0,T),
        \\
		&\rho \in L^2(0,T;V) \cap H^1 (0,T; L^2(\Omega)),
		\\
		&\mu \in L^2(0,T; V),
        \hspace{4.75em}
        F'(\rho) \in L^2(0,T; V),
	\end{align*}
	and the following identities are satisfied:
	\begin{align}
        & \langle \partial_t \rho, v \rangle + (\nabla\mu,\nabla v) = 0
        &&\hspace{-2.5cm} \forall v \in V,\qquad \text{a.e. in }(0,T),\label{phi}
        \\
        & \mu = F'(\rho) - K \star \rho 
        &&\hspace{-2.5cm} \text{a.e. in }\Omega \times (0,T),
        \notag
	\end{align}
	together with $\rho(0)=\rho_0$ in $\Omega$.
	The weak solution also satisfies the energy identity
	\begin{equation}
    \label{dissipative}
        \mathcal{F}\del[0]{ \rho(t) } + \int\limits_s^{t} \norm{ \nabla\mu(\tau) }_{L^2(\Omega)}^2 \dif\tau
        = \mathcal{F}\del[0]{\rho(s)},
        \qquad \forall\ 0\leq s\leq t<\infty,
	\end{equation}
    where $\mathcal F$ is the free energy defined in \cref{free}. Moreover, for any $\tau\in (0,T)$:
	\begin{align}
        &\sup_{t\geq \tau} \, \norm{ \partial_t \rho(t) }_{V^\prime} + \sup_{t\geq \tau} \, \norm{ \partial_t \rho(t) }_{L^2(t, t+1;  L^2(\Omega))} \leq \frac{C_0}{\sqrt{\tau}},
        \notag
        \\
        &
        \sup_{t\geq \tau} \, \norm{ \mu(t) }_{V} + \sup_{t\geq \tau} \, \norm{ \rho(t) }_{V} \leq \frac{C_0}{\sqrt{\tau}},
        \label{H1}
        \\
        &
        \norm{ F'(\rho) }_{L^\infty(\tau,t;V)} \leq C_1,
        \qquad \forall t\geq \tau,
        \notag
	\end{align}
	where the positive constant $C_0$ depends only on the initial datum energy $\mathcal F(\rho_0)$, $\overline{\rho}_0$, $\Omega$, $T$, and the parameters of the system, whereas $C_1 = C_1(\tau)$ also depends on $\tau$.
    \end{subequations}

    Finally, we have the following continuous dependence estimate: for any two weak solutions \(\rho_1\) and \(\rho_2\) to \cref{PDS:NL-CH} on the interval \([0, T]\) with \(T > 0\), corresponding to the initial data \(\rho_{0,1}\) and \(\rho_{0,2}\), respectively, it holds for all \(t \in [0, T]\) that:
    \[
        \norm{ \rho_1(t) - \rho_2 (t) }_{ V^\prime }^2 
        \leq
        \norm{ \rho_{0,1} - \rho_{0,2} }^2_{V^\prime} 
        + 
        C_2 \, \abs[1]{ \overline{\rho}_{0,1} - \overline{\rho}_{0,2} } e^{C_3 T},
    \]
    where \(C_3 > 0\) is a constant and \(C_2\) is a constant that depends on \( \norm{ F'(\rho_1) }_{L^1 \del{0,T; L^1(\Omega)}} + \norm{ F'(\rho_2) }_{L^1 \del{0,T; L^1(\Omega)}} \).
\end{theorem}

\begin{remark}
Observe that we have stated the theorem in the case of a bounded Lipschitz domain, and thus we cannot expect to apply elliptic regularity results leading to the \(H^2\)-regularity for $\mu$ (cf. with the additional regularity of \cite[Equation 3.7]{P}, in the case of smooth domains).
\end{remark}

\begin{remark}\label{Re:AlternativeAdmissibility}
    Notice that the following alternative hypotheses on \(K\) are also admissible for \(d \in \{2,3\}\) \cite[p.\,5288]{Gal2017}:
\begin{enumerate}[label=(\alph*), leftmargin=1.7em, itemsep=0.em]
	\item \( K \in W^{2,1} \del{ B(0;\xi) } \) for \( \xi \sim \text{diam} (\Omega)\) and \( \overline{\Omega} \subseteq B(0;\xi)\), and also \( K(-x) = K(x)\),  for almost any $x\in\R^d$, $d\in\{2,3\}$,
	\item \( K\in C^3(\R^d \setminus \{0\}) \cap W_{\text{loc}}^{1,1} (\R^d) \), radially symmetric \(K(x) = k\del{\|x\|} \), \(k\) is non-increasing, \(k''(r)\) and \( k'(r)/r\) are monotone in \( \R_{> 0}\), and \( |\partial^\alpha u| \leq C \|x\|^{-1} \) for \( |\alpha| = 3\).
\end{enumerate}

Option (c) is of particular interest, as it implies that \(K\) is reasonably well-behaved at the origin, in the sense that it can present, at worst, a logarithmic singularity \cite[Remark 1]{Bedrossian2011}. It can further be proven \cite[\S1.3, Lemma 2]{Bedrossian2011} that even though the second derivatives of \( K \) are not, in general, locally integrable, it is still possible to properly define \( \Delta K \star u\) as a bounded linear operator on \( L^p(\Omega)\) with \( p \in (1,\infty)\), and it involves a Cauchy principal value integral.
\end{remark}

If we now add some further assumptions on the singular potential $F$, there are many recent results allowing us to guarantee that the solution $\rho$ stays instantaneously uniformly away from the pure phases. In particular, when $d=2$, as first introduced in \cite{GP}, as $\delta\to 0^+$ we assume, for some $\beta > \sfrac 1 2$:
\begin{equation}
\label{Further_d=2}
    \frac{1}{F'(1-2\delta)} = O\del{ \, \abs{\log(\delta) }^{-\beta} }
    \qquad\text{and}\qquad
    \frac{1}{\abs{F'(-1+2\delta)}} = O\del{ \, \abs{\log(\delta) }^{-\beta} }.
\end{equation}
On the other hand, if $d=3$, we need slightly stronger assumptions. Specifically, as introduced in \cite{P}, as $\delta\to 0^+$ we assume
\begin{equation}
    \label{Further_d=3}
    \frac{1}{F'(1-2\delta)} = O\del{ \, \abs{ \log(\delta) }^{-1} },
    \qquad
    \frac{1}{F''(\pm 1 \mp 2\delta)} = O\del{ \delta },
    \qquad\text{and}\qquad
    \frac{1}{\abs{F'(-1+2\delta)}} = O\del{ \, \abs{\log(\delta) }^{-1} } .
\end{equation}
 Notice that conditions \cref{Further_d=2,Further_d=3} are verified by the logarithmic potential \cref{eq:LogPotential}. 
Thanks to the validity of \cref{prop:GN}, we can then follow word-by-word the results in \cite{GP} for the two-dimensional case or \cite{P} for the three-dimensional case. Then, we obtain the following essential result:

\begin{theorem}\label{sep1}
    Under the same assumptions of \cref{known}, assume additionally that \cref{Further_d=2} holds if $d=2$ or \cref{Further_d=3} holds if $d=3$. Then, for any $\tau>0$ there exists $\delta\in (0,1)$, which depends on $\tau$, $\abs{ \overline{\rho}_0}$, and the initial datum, such that the unique global weak solution to problem \cref{PDS:NL-CH} given in \cref{known} satisfies
    \[
        \abs{ \rho } \leq 1 - \delta 
        \qquad 
        \text{a.e. in } 
        \Omega\times(\tau,+\infty).
    \]
    In other words, the instantaneous strict separation property from the pure phases $\pm 1$ holds for all times. Notice that the quantity $\delta\in(0,1)$ can be explicitly expressed as depending on some constants related to the domain, the potential, and the parameters of the problem; see \cite{GP,P}.
\end{theorem}

We can also state some results about the convergence of any weak solution to a unique equilibrium. This convergence, under the assumption of singular potential $F$,  was shown for the two-dimensional case in \cite[Corollary 5.15]{Gal2017}, and it was recently established in the case of three-dimensional bounded domains \cite[Theorem 5.8]{P}. All the analysis has been carried out in the case of a constant mobility $m>0$. Nevertheless, in the recent \cite{GP3} the result was extended to the case of nondegenerate continuous mobilities $m(\rho)$, exploiting and extending the seminal results in \cite{GP2}. Of course, the validity of Poincaré--Wirtinger's and Gagliardo--Nirenberg's inequalities of \cref{prop:PW,prop:GN} is again essential for carrying out the arguments. For a fixed $\lambda\in(0,1)$, we introduce the set
\[
    \mathcal{H}_\lambda 
    \coloneqq  
    \cbr{ \rho \in L^\infty(\Omega) : \, \norm{\rho}_{ L^\infty (\Omega) } \leq 1, \quad \abs{ \overline{\rho} } = \lambda }.
\]
We thus obtain the following result:

\begin{theorem}\label{equi}
	Let \(\lambda \geq 0\).
	Under the same assumptions as in \cref{sep1}, suppose additionally that $F$ is real analytic in $(-1,1)$. Then, the weak solution $\rho$ from \cref{known}, departing from the initial datum $\rho_0\in \mathcal{H}_\lambda$, converges to a unique equilibrium $\rho_\infty$. 
    More precisely
    \begin{equation}
    \label{equila}
        \lim_{t\to +\infty} \norm{ \rho(t) - \rho_\infty }_{L^2(\Omega)} = 0,
	\end{equation}
    where $\rho_\infty \in \mathcal{H}_\lambda \cap V$ satisfies the \embh{stationary nonlocal Cahn--Hilliard equation}:
	\begin{equation}
        \label{conv1a}
		-K \star \rho_\infty + F'(\rho_\infty) = \mu_\infty 
        \qquad \text{in }\Omega,
	\end{equation}
	with $\mu_\infty\in \R$ as a real constant.
\end{theorem}

\begin{remark}
    Note that, for instance, the logarithmic potential $F$ in \cref{eq:LogPotential} is indeed real analytic in $(-1,1)$, and thus the additional assumption of the theorem above is satisfied.
\end{remark}

At this point, let us recall the validity of the following version of the {\L}ojasiewicz--Simon inequality:

\begin{proposition}[{\citealp[Proposition 6.2]{DellaPorta}, {\citealp[Theorem 6]{Loja}}}]\label{Lojaw}
	Let $P_0: L^2(\Omega)\to L^2_0(\Omega)$ be the projector operator from $ L^2(\Omega)$ to the subspace $L^2_0(\Omega) \coloneqq \cbr{ u\in  L^2(\Omega):\ \overline u=0 }$. Assume that $F$ satisfies \cref{assu:NLCH} and is real analytic in $(-1,1)$, $\rho\in V\cap L^\infty(\Omega)$ is such that $ \rho(x) \in [-1+\gamma,  1-\gamma]$, for any $x\in\overline{\Omega}$ and for some $\gamma\in(0,1)$, and $\rho_\infty$ satisfies \cref{conv1a}. Then, there exist $\vartheta\in \del{0,\sfrac{1}{2} }$, $\eta>0$, and a positive constant $\overline C$ such that
    \begin{equation*}
        \abs{ \mathcal{F}(\rho) - \mathcal{F}(\rho_\infty) }^{1-\vartheta} 
        \leq
        \overline C\,
        \norm{ P_0\del{ F'(\rho) - K \star \rho } }_{V^\prime},
    \end{equation*}
\end{proposition}
    whenever \( \norm{\rho - \rho_\infty}_{L^2(\Omega)} \leq \eta \).

In addition to the convergence to stationary states results, we aim to find a rate of convergence for \cref{equila} as $t\to \infty$. This is indeed possible after applying the {\L}ojasiewicz--Simon inequality, since the solution is strictly separated from any positive time onwards. The result reads as follows:

\begin{lemma}\label{Conv_Rate}
    Under the same assumptions of \cref{equi}, the weak solution $\rho$ from \cref{known}, departing from the initial datum $\rho_0\in \mathcal{H}_\lambda$, converges to a unique equilibrium $\rho_\infty$, satisfying \cref{equila}, with polynomial decay rate in the $L^2$ distance:
	\[
        \norm{ \rho(t) - \rho_\infty }_{L^2(\Omega)}
        \leq C (1+t)^{ -\frac{\vartheta}{2(1-2\vartheta)} }
	   \qquad \forall t\geq 1,
	\]
    where $C>0$ and $\vartheta\in (0,\sfrac{1}{2})$ are suitable constants depending on the initial data, $\Omega$, $F$, and the other parameters of the system.
\end{lemma}
\begin{proof}
    By reproducing the proof of \cite[Theorem 5.8]{P}, let us fix $\vartheta\in \del{0,\sfrac{1}{2} }$ and $\eta>0$ given in \cref{Lojaw}, where we choose $\gamma$ equal to the value of $\delta$ given in \cref{sep1}. 
    It is then possible to show that there exists $t^*>0$ such that $\norm{ \rho(t) -\rho_\infty }_{ L^2(\Omega)} \leq \eta$, for all $t\geq t^*$, where $\rho_\infty$ is defined in \cref{equi}. Therefore, since the solution $\rho$ enjoys the separation property by \cref{sep1} and thanks to the choice of $\gamma$, by \cref{Lojaw} we get, for any $t\geq t^*$, that
	\begin{align}
        \label{LSineq}\del{ \mathcal{F}(\rho) - \mathcal{F}(\rho_\infty) }^{1-\vartheta} 
        \leq
        \norm{ P_0\del{ F'(\rho) - K \star \rho } }_{V^\prime} 
        \leq
        \overline C\, \norm{ P_0\mu}_{L^2(\Omega)} 
        \leq
        \widehat{C} \, \norm{ \nabla \mu }_{ L^2(\Omega)},
	\end{align}
    where $\widehat{C} > 0$ depends on $\overline C$ and on the Poincaré--Wirtinger constant; see \cref{prop:PW}. Recall that $\mathcal F(\rho(\cdot))$ is nonincreasing in time, and thus $\mathcal F(\rho(t))-\mathcal F(\rho_\infty)\geq 0$ for any $t\geq0$. We can actually assume the strict positivity of this quantity, as otherwise, if there exists $t_1>0$ such that $\mathcal F(\rho(t_1))=\mathcal F(\rho_\infty)$, then $\rho$ would be invariant in time for $t\geq t_1$, and the proof is finished. As a consequence, thanks to the energy identity \cref{dissipative} and using \cref{LSineq}, we deduce, for any $t\geq t^*$, that
    \begin{align}
        -\dod{ }{t}     \del{ \mathcal{F}(\rho) - \mathcal{F}(\rho_\infty) }^{\vartheta} 
        &= -\vartheta \del{ \mathcal{F}(\rho) - \mathcal{F}(\rho_\infty) }^{\vartheta - 1} \dod{ }{t} \mathcal{F}(\rho)
        \notag
        \\
        &\geq
        \del{ \mathcal{F}(\rho) - \mathcal{F}(\rho_\infty) }^{\vartheta - 1} \vartheta \norm{ \nabla \mu }^2_{L^2(\Omega)} 
        \geq
        C_\text{A} \del{ \mathcal{F}(\rho) - \mathcal{F}(\rho_\infty) }^{1 - \vartheta},
        \label{i1}
    \end{align}
    where $C_\text{A}>0$ is a positive constant independent of $t$. By \cref{LSineq}, we get
    	\begin{align*}
        \norm{ \nabla \mu }_{ L^2(\Omega)}^2\geq \frac1{\widehat C^2}\del{ \mathcal{F}(\rho) - \mathcal{F}(\rho_\infty) }^{2-2\vartheta}.
	\end{align*}
    Also, we analogously infer that
    \[
        -\dod{ }{t}     \del{ \mathcal{F}(\rho) - \mathcal{F}(\rho_\infty) }^{\vartheta} 
        = -\vartheta \del{ \mathcal{F}(\rho) - \mathcal{F}(\rho_\infty) }^{\vartheta - 1} \dod{ }{t} \mathcal{F}(\rho)
        \geq
        C_\text{B} \norm{ \nabla \mu }_{ L^2(\Omega) },
    \]
    where $C_\text{B} > 0$ is again a positive constant independent of $t$.
    We first immediately see from \cref{i1} that 
    \begin{equation}
        \label{estf}
        \mathcal{F}\del{\rho(t)} - \mathcal{F}(\rho_\infty) 
        \leq
        C_\text{C} (1+t)^{ -\sfrac{1}{(1-2\vartheta)} }
        \qquad \forall t\geq t^*,
    \end{equation}
    for some $C_\text{C} > 0$.
Indeed, it is enough to set $z(t):=\mathcal F(\rho(t))-\mathcal F(\rho_\infty)>0$ to obtain from \cref{i1} that
\begin{align*}
    \dod{ }{t} z(t)^\vartheta + C_\text{A} z(t)^{1-\vartheta}\leq 0,
\end{align*}
    and thus, since $z(t)>0$, 
    \begin{align*}
    \vartheta z(t)^{2(\vartheta-1)}\dod{ }{t} z(t) + C_\text{A} \leq 0,
\end{align*}
entailing
    \begin{align*}
    \frac{\vartheta}{2\vartheta-1} \dod{ }{t}z(t)^{2\vartheta-1} + C_\text{A} \leq 0,
\end{align*}
so that, since $\vartheta\in(0,\tfrac12)$, we infer
    \begin{align*}
     \dod{ }{t}z(t)^{2\vartheta-1}\geq C_\text{A} \frac{1-2\vartheta}{\vartheta},
\end{align*}
which can be integrated over $(t^*,t)$, $t\geq t^*$, giving 
  \begin{align*}
     z(t)\leq \left(z(t^*)^{2\vartheta-1}+{C_\text{A}}\frac{1-2\vartheta}{\vartheta}(t-t^*)\right)^{-\frac 1{1-2\vartheta}}.
\end{align*}
Since $t^*>0$ is fixed, this implies \cref{estf}, after rearranging the constants.

 By comparison in \cref{phi}, recall that the following inequality holds:
    \begin{equation*}
        \int\limits_t^\infty \norm{ \partial_t \rho }_{V^\prime} \dif s 
        \leq
        \int\limits_t^\infty \norm{\nabla \mu}_{L^2(\Omega)} \dif s.
    \end{equation*}
    Integrating \cref{i1} over $(t,\infty)$, for $t\geq t^*$, where we recall that $\lim\limits_{s\to \infty} \mathcal F\del{\rho(s)}  = \mathcal F(\rho_\infty)$, and using \cref{estf}, we then obtain
    \[
        \norm{\rho(t)-\rho_\infty}_{V'} 
        \leq 
        \int\limits_t^\infty \norm{\partial_t\rho}_{V^\prime} \dif s
        \leq  \tilde C(1+t)^{-\sfrac{\vartheta}{(1-2\vartheta)}},
    \]
    for some $\tilde C>0$.
    Additionally, since $\rho(t)$ is uniformly bounded in $V^\prime$ for any $t\geq0$, we also have, up to changing constants,
    \[
        \norm{\rho(t)-\rho_\infty}_{V^\prime} 
    \leq 
    \tilde {\tilde C}(1+t)^{-\sfrac{\vartheta}{(1-2\vartheta)} }
    \qquad \forall t\geq 1,
    \]
    for $\tilde {\tilde C}>0$.
    Therefore, using standard interpolation results (see, e.g., \cite[Proposition 2.1]{LionsMagenes}),
    we infer that there exists $C>0$ such that 
    \[
        \norm{\rho(t)-\rho_\infty}_{L^2(\Omega)}
        \leq 
        C_I \norm{\rho(t)-\rho_\infty}_{V^\prime}^{\sfrac{1}{2}} 
        \norm{\rho(t)-\rho_\infty}_{V}^{\sfrac{1}{2}}
        \leq  
        C(1+t)^{-\sfrac{\vartheta}{2(1-2\vartheta)}}
        \qquad \forall t\geq 1,
    \]
where $C_I>0$ and we exploited the uniform bound in $V$ for $\rho$, for any $t\geq1$, given by \cref{H1}.
\end{proof}

\subsection{Choice of the model}

As stated in \cite{Frigeri2021}, the numerical simulation, analysis, and optimal control of problems involving the nonlocal and singular potentials have not received much attention in the literature. This work aims to provide a computational approach to determine solutions of the nonlocal Cahn--Hilliard system where a logarithmic potential \(F\) is employed alongside a convolution kernel \(K\) displaying a singularity at the origin. 
Due to the computational challenges of the convolution step, we will lay out the setup for a pseudospectral method.

\subsubsection{Domain selection}

Let \( \Omega \) be the unit square \( \Omega \coloneqq B_\infty (\sfrac{1}{2} \Ones[2];\sfrac{1}{2}) = (0,1) \times (0,1)\) and \(T>0\), and recall the definition \( Q \coloneqq \Omega \times (0,T)\). Observe that by definition \(\Omega\) is a Lipschitz domain that satisfies the cone condition, and the results derived from \cref{prop:PW} in \cref{sec:CH_WP} follow. 

Notice that any other domain diffeomorphic to \(\Omega\) can be easily adapted to our methodology; for examples, see \cref{app:diffs}. To see this, suppose that there exists a Lipschitz domain \(\Theta\) and a map \( \Psi: \overline{\Omega} \mapsto \overline{\Theta} \) with \(\Psi\) invertible, with a smooth inverse. 
Consider a point \( y \in \Theta\) and its corresponding preimage \( x = \Psi^{-1}(y)\), then
\begin{align}
	(K \star \rho) \, (y) &= \int\limits_\Theta K(y-u) \rho(u) \dif u 
	=  \int\limits_{\Psi (\Omega)} K(y-u) \rho(u) \dif u 
	=  \int\limits_\Omega K( y - \Psi s ) \, \rho\del{\Psi s}  \abs{ \det\del{J_\Psi (s)} } \dif s		\notag
	\\
	&= \int\limits_\Omega K( \Psi x - \Psi s ) \, \rho\del{\Psi s}  \abs{ \det\del{J_\Psi (s)} } \dif s
	=  \del{ (K \circ \Psi) \star \abs{ \det\del{J_\Psi} } (\rho \circ \Psi) } (x),\label{eq:Conv_Change_of_Vars}
\end{align}
where \(J_\Psi\) is the Jacobian of \(\Psi\). This yields an algorithm for computing a discretisation of the convolution of \(K\) with any function \(u\) over \(\Theta\). Suppose that \(u\) is pointwise discretised over a mesh yielding a vectorial representation \(\bfu\). The convolution matrix, labelled \( C_{K,\Theta} \in \mathcal{M}_n(\R)\), is constructed from the discretisation of a different convolution kernel \(G\) defined over \(\Omega\). Defining \( G \coloneqq K \circ \Psi\) and \( \mathbf{j} \) as the vector  of determinants of \(J_\Psi\) evaluated at the same collocation mesh as \(\bfu\). By virtue of \cref{eq:Conv_Change_of_Vars}, we have that
\(
	K \star u  \approx C_{K,\Theta} \bfu = C_{K\circ \Psi,\Omega} \diag(\mathbf{j}) \, \bfu.
\)
As a result, we only need a procedure to compute the action of a general convolution kernel for functions defined in \(\Omega\). 
This is a standard procedure in pseudospectral methods, where we can define most pseudospectral structures for a fixed domain, which can in turn be associated with other domains via a combination of the chain rule and a conformal map \cite{Nold_2017,Roden2022a}.

We will focus most of our analysis and tests on \(\Omega\); however, we will further exemplify two additional domains. Specifically, these are the unit ball \( B_2( 0; 1 ) \eqqcolon \bigcirc \) and the three dimensional box 
\(  B_\infty (0;1)  \eqqcolon \text{\faCube} \).

\subsubsection{Energy potential}

We recall the logarithmic potential from \cref{eq:LogPotential}:
\(
	F(s) = \frac{\theta}{2} \sbr{ (1+s) \log (1+s) + (1-s) \log (1-s)  }.
\)
Notice that the second hypothesis in \cref{assu:NLCH} is satisfied by \(F\). Additionally, from \cref{eq:DerivativesOfF}, we have that \cref{Further_d=2,Further_d=3} hold as well.

The interest in potentials like \(F\) arises from the fact that certain physical quantities such as tumour mass densities are only defined if the phase field variable \(\rho\) belongs to the physical interval \( [-1,1]\), which cannot be guaranteed for other choices like the double-well potential \cite{Frigeri2021}; see also \cite[\S 1]{Miranville2010}.

\subsubsection{Kernel selection}\label{ssec:Kernel_Selection}

We will mostly employ the Newtonian potential:
\begin{equation}\label{eq:Newtonian_potential}
	K: \R^2 \setminus\{0\} \ni x  \longmapsto  \frac{1}{2\pi} \log \|x\| = \frac{1}{4\pi} \log ( x_1^2 + x_2^2 ) \in \R.
\end{equation}
Notice that \(K\) is differentiable; its gradient is given by
\(
	\nabla K(x) = \frac{1}{2\pi \|x\|^2}  \begin{bsmallmatrix} x_1 & x_2 \end{bsmallmatrix}^\top \hspace{-0.3em}
\).
Additionally, we have that
\begin{align*}
	\cbr{ \partial^\alpha K: |\alpha| = 2 } 
	&\subseteq \frac{1}{2\pi \|x\|^4} \cbr{ \pm(x_1-x_2)(x_1+x_2) , -2x_1 x_2 },
	\\
	\cbr{ \partial^\alpha K: |\alpha| = 3 } 
	&\subseteq \frac{1}{2\pi \|x\|^6} \cbr{ \pm 2x_1 (x_1^2 - 3x_2^2), \pm 2x_2 (3x_1^2 - x_2^2) }.
\end{align*}
%

Moreover, \(-K\) is the fundamental solution of the Poisson equation in \(\R^2\); i.e., \( v = -K \star f\)  satisfies \( -\Delta v = f\) for \( f \in C^2_c (\R^2)\). Alternatively, we can formally write \( \Delta K = \delta_0\). For a more detailed discussion of this construction, we refer the reader to \cite[\S 2.2]{Evans_2010} and \cite[\S2.4 \& Chapter 4]{Gilbarg2001}.
Observe that \(K\) is also radially symmetric with \( k = \log\), and the other hypotheses for condition (c) on the alternative admissibility in \cref{Re:AlternativeAdmissibility} follow as consequence of the following result:

\begin{lemma}[{\citealp[Lemma 18.2.1]{AMT_Th_2024}}]
\label{lem:Suitability_of_Kernel}
	The functions \(K\) and \(\nabla K\) are continuous, and the singularity at the origin is well defined in terms of singular values, thus locally integrable in \( \R^2\) and \(K \in W_{\text{loc}}^{1,1}(\R^2)\).
\end{lemma}
\section[A spectral element discretisation for the Newtonian potential]{A spectral element discretisation  for the Newtonian potential}
\label{ch:Spectral_Element_Newtonian_Potential}

Let \(K\) be a kernel satisfying any of the assumptions from the previous section. 
The main objective of the present section is to develop a method to approximate \( K\star \rho\) via the cutoff decomposition of the integral; i.e., for \(\varepsilon>0\), we write
\begin{align}
	(K \star \rho) \, (x) &= \int\limits_{\square} K(x-y) \, \rho(y) \dif y
	\label{eq:Newtonian_Volume}
	\\
	&=  
	\int\limits_{\dsquare}  K(x-y)\, \rho(y) \dif y + 
	\int\limits_{\smallblacksquare} K(x-y)\, \rho(y) \dif y 
	\eqqcolon \mathcal{I}_1[\rho](x) + \mathcal{I}_2[\rho](x) ,
	\label{eq:Newtonian_Decomposition}
\end{align}
where we have introduced the graphical notation
\( \Omega \eqqcolon \square\), \( \Omega \setminus B_\infty (x; \varepsilon)  \eqqcolon \dsquare \), and \( \Omega \cap B_\infty (x; \varepsilon)  \eqqcolon \scaleobj{0.75}{\blacksquare} \). 

The integral term \( \mathcal{I}_1\) does not contain any singularity and thus can be computed by any quadrature method\footnote{In fact, notice the principal value relation
\[
	(K \star \rho) \, (x) = \lim_{\varepsilon \downarrow 0} \int\limits_{\dsquare} K(x-y) \, \rho(y) \dif y.
\]
}.
For the second term, we will use the zero-order Taylor expansion \cite{Cavaliere2014}, also known as a piecewise constant approximation:
\begin{equation}\label{eq:Constant_Approximation_NP}
	\mathcal{I}_2[\rho](x)
	=
	\int\limits_{x - \smallblacksquare} K(u)\, \rho(x-u) \dif u \approx 
	\int\limits_{x - \smallblacksquare} K(u)\, \rho(x)  \dif u =  \rho(x) \int\limits_{x - \smallblacksquare} K(u)  \dif u,
\end{equation}
which is reasonable (as we will see) when \(\varepsilon\) is close to zero due to the region of self-interaction. Here, the Minkowski difference \(x - \scaleobj{0.75}{\blacksquare} \) emerges as a result of the change of variable \( y \equiv x-u\).

The use of a cutoff or truncation approximation is a standard technique in real analysis \cite{Grafakos2014} and in the discretisation of singular integro-differential operators under the presence of a weak (i.e., integrable) singularity. In \cite[Chapter 7]{Li2019}, Taylor approximations are employed along with quadrature methods to overcome the singularity for fractional Laplacians. This approach has further been extended to the variable exponent case \cite{Alzahrani2022,Lei2023} based on the methods developed in \cite{Pozrikidis2016,Minden2020}. Different quadrature-based ﬁnite diﬀerence discretisations are introduced in \cite{Tian2013} for the study of peridynamic models. Later, in \cite{Tian2015}, the authors introduced a nonconforming Galerkin method that replaces the singularity by a constant value approximation. Calculus of differential forms can be used to convert the area integral to a suitable boundary integral. A review with applications in optimal control is presented in \cite{Zhu2021}, and collocation-based approaches are discussed in \cite[Chapter 12]{Martinsson2019}. For uniform meshes, general-purpose methods are available using the Fourier transform and its fast variants; see \cite{Guan2014a,Vico2016,Guan2020,Tiwari2022,Klinteberg2024} and the references therein. Many of the aforementioned methods, extensions, improvements, and alternative approaches are discussed in the thorough monograph \cite{D_Elia2020a}. Other available methods are the Discrete Singular Convolution algorithm \cite{Wei1999,Zhao2003} and wavelet transform methods \cite{Panja2020,Assari2018,Wang2024,Gantumur2005}.

%
%

In the context of the Newtonian potential \cref{eq:Newtonian_potential}, one of the main uses of this function is to determine a particular solution of Poisson's equation. This particular solution can be obtained by direct evaluation of the integral \cref{eq:Newtonian_Volume}. To name some works in this direction, in \cite{Anderson2023} singular and near-singular quadrature is handled by converting integrals on volumetric regions to line integrals using a dilation integral representation; a similar procedure is used in \cite{Shen2022} where a coordinate transformation decomposes the self-interacting region, which is then integrated using a generalised Gaussian quadrature; these methods are based on a differential form approach \cite{Zhu2021,NintcheuFata2012}. Alternatively, the potential can be evaluated over a larger and regular domain \( \Omega_+ \supset \Omega\) for an extended density function \(\rho_+\) such that \( \rho_+ |_\Omega = \rho|_\Omega  \). This is often based on the Fast Multipole Method \cite{Askham2017} (see also \cite[Chapter 7]{Martinsson2019}) or the Fast Fourier Transform \cite{Tiwari2022}.

Spectral methods have also been used in this context; see \cite[\S 8--9]{D_Elia2020a}. The fractional Laplacian has been discretised in Fourier space in \cite{Antil2017}. The classical Cahn--Hilliard system was discretised using a pseudospectral method in \cite{Chen1998} and many variants have been analysed; see \cite{Li2016} and the references therein. Later on, in \cite{Ainsworth2017}, a fractional Cahn--Hilliard system (where the divergence form is replaced with a fractional Laplacian and periodicity is assumed at the boundary) was studied using the Fourier--Galerkin method. A similar system was considered in \cite{Weng2017} (where the chemical potential term involves a fractional Laplacian instead jointly with Dirichlet and Neumann conditions), where a spectral decomposition was considered. A time fractional variant with Dirichlet conditions was also covered in \cite{Du2020} with spectral elements in space.
More recently, viscous Cahn--Hilliard models with non-negative nonlocal diffusion operators have been studied using the Fourier pseudospectral method combined with operator splitting techniques; see 
\cite{Chen2024a,Zhai2021a}.
Finally, let us mention that some recent work on stabilised Fourier spectral schemes for the case of regular potentials, see \cite{Li2023a} and the references therein.

To the best of our knowledge, the numerical study of nonlocal systems like \cref{PDS:NL-CH} (namely, the nonlocal Cahn--Hilliard system with a weakly singular kernel, a singular potential, and no-flux boundary conditions) is scarce in the literature, and the present work provides the first such computation. Following the spirit of \cite{Tian2013}, we use the approximation \cref{eq:Constant_Approximation_NP} to develop a simple, easy-to-follow, and reproducible numerical procedure for \cref{eq:Newtonian_Volume}. This task is carried out in two steps. First, we explicitly compute the integral approximation \(\mathcal{I}_2\) in \cref{eq:Constant_Approximation_NP}, which serves both as a correction to \(\mathcal{I}_1\) and as an error indicator for the method. Second, we employ a spectral element method to precompute a matrix representation of \(\mathcal{I}_1\).

\subsection{The Newtonian potential over a square region}
\label{sec:CH_Newtonian_Exact_Formulae}

In this section, we focus on the evaluation of the integral term at the end of \cref{eq:Constant_Approximation_NP}. This task requires a rigorous definition of the convolution operator. Convolution is often defined over the full space \cite[\S 4.4]{Brezis_2011} due to its relationship with topological groups \cite[\S 1.2]{Grafakos2014}, yet this is a challenge for our framework as we can only determine the values of \(\rho\) from \cref{PDS:NL-CH} which, in turn, provides information up to the boundary (in the sense of traces) of the unit square\footnote{In the following discussion, we will omit the fact that \(\rho\) depends on time and will treat it as a variable that depends merely on space.}. As a result, we can introduce the zero extension operator defined in terms of the indicator function of \(\Omega\):
\(
	\mathcal{E}_1 : L^p(\Omega) \ni \rho  \longmapsto 
	\rho \mathsf{1}_\Omega 
	\in L^p(\R^2)
\).
Since the boundary of \(\Omega\) is a set of zero measure in \(\R^d\), we have that \(\mathcal{E}_1\) is an isometric embedding. This allows us to define, for any \( K \in L^{1}_{\text{loc}} (\R^2)\) and \( x\in \Omega\):
\begin{equation}
\label{eq:Convolution_CH_Extension}
	(K \star \rho) (x) \coloneqq 
	\int\limits_{\R^d} K(x-y) \, \mathcal{E}_1[\rho](y) \dif y 
	= \int\limits_{\Omega} K(x-y) \, \rho(y) \dif y .
\end{equation}
The question of whether \cref{eq:Convolution_CH_Extension} is well defined requires the introduction of a second extension operator 
\(
	\mathcal{E}_2 : L_{\text{loc}}^1(\R^2) 
	\ni K \longmapsto K \mathsf{1}_\Xi 
    \in 
	L^1(\R^2)
\), 
where \( \Xi \subsetneq \R^2\) is the support of the convolution for the difference inside the kernel in \cref{eq:Convolution_CH_Extension}. Specifically, \( \Xi \) is the Minkowski difference of \(\Omega\) with itself, which in this case coincides with the unit \(\ell_\infty\) ball; i.e., 
\(
	\Xi \coloneqq \Omega - \Omega = B_\infty(0;1)
\).
Consequently, Young's convolution inequality \cite[Theorem 4.15]{Brezis_2011} allows us to determine that 
the estimate 
\(
	\| K \star \rho \|_{L^p(\Omega)} \leq \| K \|_{L^1(\Xi)} \| \rho \|_{L^p(\Omega)}
\)
holds for any \( p \in [1,\infty]\). Finally, for any solution of \cref{PDS:NL-CH}, we have that
\( \| K \star \rho \|_{L^\infty(\Omega)} \leq \| K \|_{L^1(\Xi)} \).

%

The computation of \cref{eq:Constant_Approximation_NP} will be done in three stages. First, we compute a simplified version inside a subset of \(\Omega\). Second, we use a closed formula to derive a general result regarding the integral of \(K\) over \(\Xi\). Lastly, we combine the steps used to compute these results to achieve the required quantity.

\subsubsection{Integration in the positive quadrant}

Although the appropriate domain to understand the convolution in \cref{eq:Newtonian_Volume} is the unit square ball centred at the origin, let us first assume that \(K\) is nonzero only within \( \Omega\). 
Employing the graphical notation introduced at the beginning of this section, we have that the difference \(x-y\) is only included in \(\square\) for \(y \in [0,x_1] \times [0,x_2]\), hence we introduce the positive quadrant convolution operator \( \star_+\) as
\begin{equation}\label{eq:1stQuadrant_Int}
	(K \,\star_+ \rho)\, (x) 
	= \int\limits_{\square} \mathcal{E}_1 [K] (x-y) \, \rho(y) \dif y 
	= \int\limits_0^{x_1} \int\limits_0^{x_2} K (u) \, \rho(x-u) \dif u.
\end{equation}
%

\begin{lemma}\label{lem:I_analytical_expression}
	Let \( \mathtt{I}: \overline{\square} \ni x \longmapsto (K \, \star_+ \mathsf{1}_\square) (x)  \in \R\), where the convolution operator is defined as in \cref{eq:1stQuadrant_Int}. Then, \(\mathtt{I}\) admits the closed-form expression
	\begin{align*}
		\mathtt{I}(x) = 
		\frac{x_1 x_2}{4\pi} \big[ 2 \log \|x\| - 3 \big] 
		+ \frac{1}{16} \|x\|^2	
		- \frac{1}{8\pi} ( x_2^2 - x_1^2 ) 	\operatorname{atan2}\del{x_2^2 - x_1^2, 2 x_1 x_2}.
		%
		%
	\end{align*}
	Furthermore, \( \mathtt{I}\) is continuous, symmetric, and nonpositive on the closure \( \overline{\square}\), with \( \mathtt{I}(x) = 0\) along the truncated coordinate cross \(  \{0\} \times [0,1] \cup [0,1] \times \{0\} \). A global minimum is attained at \( x = s \Ones[2] \), where \( s = (\sfrac{\sqrt 2}{2}) \exp(1-\sfrac{\pi}{4}) \). 
\end{lemma}

The proof of \cref{lem:I_analytical_expression} follows by splitting \(\square\) into two triangular regions represented in polar coordinates: one below the hypotenuse of the triangle given by the origin, $x$, and $(0,x_1)$, and the other being its complement in $\square$; i.e., the triangle with vertices in the origin, $x$, and $(x_2,0)$. The associated integral quantities have exact primitives which can then be expressed back in Cartesian coordinates. Symmetry follows from the fact that \( \R^2 \ni x \mapsto \operatorname{atan2}(x_2,x_1) \to [-\pi,\pi] \) gives the angle (in radians) between the positive horizontal axis and the vector \(x\), ensuring the correct quadrant is considered. For a complete detailed derivation, we refer the reader to \cite[\S 19.1.1]{AMT_Th_2024}.

\subsubsection{Integration in the whole domain}\label{ssec:NLCH_Int_Conv_Whole}

Consider once again \cref{eq:Convolution_CH_Extension} with \(K\) defined in the whole of \(\Xi\). Under the change of variables \( u \equiv x - y\):
\[
	(K \star \rho) \,(x) 
	= \int\limits_{\square} \mathcal{E}_2 [K] (x-y) \, \rho(y) \dif y 
	= \int\limits_{x-\square} K (u) \, \rho(x-u) \dif u
	 = \int\limits_{x_2-1}^{x_2} \int\limits_{x_1-1}^{x_1} K (u) \, \rho(x-u) \dif u.
\]
If \(\rho\) is a constant or radially symmetric (recall that our aim is to use the local approximation \cref{eq:Constant_Approximation_NP}), we can then decompose \( x - \square\) into four integration regions in the positive quadrant. Thus, we obtain the following result:

\begin{lemma}[{\citealp[\S 19.1.2]{AMT_Th_2024}}]
\label{lem:J_analytical_expression}
	Let \( \mathtt{J}: \overline{\square} \ni x \longmapsto K \star \mathsf{1}_\square (x)  \in \R\), where the convolution operator is defined as in \cref{eq:Convolution_CH_Extension}. Then, \(\mathtt{J}\) admits a closed-form formula in terms of \( \mathtt{I}\):
\begin{equation}
\label{eq:CH_Conv_Whole_Domain}
	\mathtt{J}(x) = 
	\begin{cases}
		\mathtt{I}(x_1,x_2) + \mathtt{I}(1-x_1,x_2) + \mathtt{I}(x_1,1-x_2) + \mathtt{I}(1-x_1,1-x_2)
			& \text{if } x \in \square,
		\\
		\mathtt{I}(1,s) + \mathtt{I}(1,1-s) 	& \text{if } x_s \in \partial \square, 
	\end{cases}
\end{equation}
	where \( s\) in the last expression is the entry of \(x\) not in the set \( \{0,1\}\) if \(x\) is not a corner, else \( s = 0\). 

	Additionally, \( \mathtt{J}\) is continuous, symmetric, and strictly negative on the closure \( \overline{\square}\). There are four global maxima at the corners of \(\overline{\square}\) and there is a global minimiser at the centre of the domain.
\end{lemma}

Combining \(\mathtt{I}\), \(\mathtt{J}\), and the fact that \(K\) is negative inside the unit ball centred at the origin yields:

\begin{corollary}
	The convolution against \(K\) is a contraction in \(L^p(\square)\). Specifically, it holds that
	\(
		\|K\|_{L^1(\Xi)} 
		< 1
	\).
\end{corollary}

\subsubsection{Integration over a small neighbourhood}

To evaluate \( \mathcal{I}_2[\rho] \), as introduced in \cref{eq:Constant_Approximation_NP}, we now focus on computing the integral factor that encodes the interaction between the kernel and the geometry near \( x \). This quantity is defined by
\begin{equation}\label{eq:G_Diagonal_Approximation_Newtonian}
	\mathtt{G}_\varepsilon(x) 
	\coloneqq
	\int\limits_{x - \smallblacksquare} K(u) \dif u,
\end{equation}
where \( \scaleobj{0.75}{\blacksquare} = \Omega \cap B_\infty (x; \varepsilon) \) and \( \varepsilon \leq \sfrac{1}{2} \). The function \( \mathtt{G}_\varepsilon \) depends only on the position \( x \), the shape of the kernel \( K \), and the intersection geometry.

In analogy with \(\mathtt{J}\), it is possible to derive a closed-form expression for \(\mathtt{G}_\varepsilon\) in terms of \(\mathtt{I}\).
For convenience, we introduce the binary\footnote{%
	If \(a\) and \(b\) are vectors, we overload this operator by applying it entrywise in a broadcastable fashion. For instance, if \(a \in \R^2\) and \(b\) is a constant, then \( a\sqcap b = (a_1 \sqcap b, a_2 \sqcap b)^\top \hspace{-0.2em} \). 
	}
operator \( a \sqcap b \coloneqq \min \{a,b\}\). 

\begin{lemma}[{\citealp[\S 19.1.3]{AMT_Th_2024}}]
\label{lem:G_analytical_expression}
	Let \(x \in \overline\square\) and select a neighbourhood size \(\varepsilon \leq \sfrac{1}{2}\). The Minkowski difference between \(x\) and the portion of \( B_\infty (x; \varepsilon) \) inside \(\square\) is given by 
	\(
		x - \scaleobj{0.75}{\blacksquare} =
		\del{-(1-x_1 \sqcap \varepsilon), x_1 \sqcap \varepsilon}
		\times 
		\del{-(1-x_2 \sqcap \varepsilon), x_2 \sqcap \varepsilon}.
	\)
	Thus, \(\mathtt{G}_\varepsilon\) admits a closed-form formula in terms of \( \mathtt{I}\):
\begin{equation}\label{eq:CH_Conv_Neighbourhood_eps}
	\mathtt{G}_\varepsilon (x) = 
	\mathtt{I}\del{ x \sqcap \varepsilon }
	+
	\mathtt{I}\del{ (1-x_1, x_2) \sqcap \varepsilon }
	+
	\mathtt{I}\del{ (x_1,1-x_2) \sqcap \varepsilon }
	+
	\mathtt{I}\del{ (\Ones[2]-x) \sqcap \varepsilon }.
\end{equation}
	Additionally, \( \mathtt{G}_\varepsilon \) is strictly negative, symmetric, and there are four global maxima at the corners of \( \overline{\square}\). 
\end{lemma}

\subsubsection{Quality of the approximation}

For any neighbourhood radius \(\varepsilon \leq \sfrac{1}{2}\), let us define \( \widehat{\mathcal{I}}_2[\rho](x;\varepsilon)  \coloneqq \rho(x) \mathtt{G}_\varepsilon(x)\) for all \( x \in \overline{\square}\).
Combining the previous lemmas, we arrive at the main result of this section:

\begin{theorem}\label{th:QualityApprox}
	Let \(\rho \in {L^\infty(\square)} \) be bounded inside \([0, \alpha]\) for some \(\alpha > 0\). 
	Additionally, let \(x \in \overline\square\) and select a neighbourhood size \(\varepsilon \leq \sfrac{1}{2}\). The expression
	\begin{align*}
	(K \star \rho) \, (x) &\approx 
	\mathcal{I}_1[\rho](x) + \widehat{\mathcal{I}}_2[\rho](x;\varepsilon) 
	=  
	\int\limits_{\dsquare}  K(x-y)\, \rho(y) \dif y + 
	\rho(x) \mathtt{G}_\varepsilon(x),
	\end{align*}
	provides a \( O( \varepsilon^2 \log \varepsilon ) \) pointwise approximation of \cref{eq:Newtonian_Volume}, which is exact for constant \(\rho\).
\end{theorem}
\begin{proof}
	The result follows by construction. For the error term, notice that there is a region of global minima determined by the closure of the ball \( B_\infty( \sfrac{1}{2} \Ones[2]; \sfrac{1}{2}-\varepsilon )\). To see this, let \( x \) be an element of this set, then 
	its entries lie in the interval \( [\varepsilon, 1-\varepsilon] \) for which \cref{eq:CH_Conv_Neighbourhood_eps} reduces to
	\begin{equation}
		\label{eq:CH_Conv_Neighbourhood_eps_min}
		\mathtt{G}_\varepsilon (x) = 
		4\mathtt{I}(\varepsilon, \varepsilon)
		=
		\frac{\varepsilon^2}{2\pi} \big[ 4 \log \varepsilon + \log 4 - 6 + \pi \big]
		\sim O( \varepsilon^2 \log \varepsilon ).
	\end{equation}	
	
	Observe that if \( u \in x - \scaleobj{0.75}{\blacksquare}\), then \( \|u\|_2 \leq \sqrt{2} \varepsilon < 1  \), then \( \|K\|_{L^1 (x - \scaleobj{0.75}{\blacksquare})} = \abs{ \mathtt{G}_\varepsilon(x)} < 1  \). Then
	\begin{align*}
		\abs{ \mathcal{I}_2[\rho](x) - \widehat{\mathcal{I}}_2[\rho](x;\varepsilon) }
		&
		\leq 
		\int\limits_{x - \smallblacksquare} \abs{ K(u) } \abs{ \rho(x-u) - \rho(x) } \dif u
		\leq 
		\|K\|_{L^1 (x - \scaleobj{0.75}{\blacksquare})}	
		\operatorname*{ess\,sup}_{y \in \smallblacksquare} \abs{ \rho(y) - \rho(x) }.
	\end{align*}
	The bounds on \(\rho\) and the sign of \( \mathtt{G}_\varepsilon\) yield
	\begin{equation}\label{ineq:Pointwise_Bound}
		\abs{ \mathcal{I}_2[\rho](x) - \widehat{\mathcal{I}}_2[\rho](x;\varepsilon) } \leq -\alpha \mathtt{G}_\varepsilon (x) .
	\end{equation}
	Maximising \cref{ineq:Pointwise_Bound} over \(\square\) and recalling \cref{eq:CH_Conv_Neighbourhood_eps_min} yields the final error estimate.
\end{proof}

The absolute value of \cref{eq:CH_Conv_Neighbourhood_eps_min} decreases as \(\varepsilon\) becomes smaller. In fact, dividing \(\varepsilon\)  by \(10\) approximately reduces the value of \( |\mathtt{G}_\varepsilon | \) by a factor of \(100\). 
%
%
Thus \( |\mathtt{G}_\varepsilon | \) already reaches machine precision for \( \varepsilon \lesssim 10^{-8}\).



Our first set of numerical experiments will test the validity of the approximation
\begin{equation}\label{eq:Conv_Singular_Approximation_Scheme}
	(K \star \rho) \, (x) \approx
	\int\limits_{\dsquare}  K(x-y)\, \rho(y) \dif y + 
	\rho(x) \mathtt{G}_\varepsilon(x),
\end{equation}
where the first integral will be computed numerically using a spectral element method, which is introduced in the next section. 
The reader may notice that 
\cref{eq:CH_Conv_Neighbourhood_eps_min} yields the additional result that the approximation scheme \cref{eq:Conv_Singular_Approximation_Scheme} displays an error that behaves like \(\varepsilon^2\) and is mesh-independent. As a result, for small values of \(\varepsilon\), we could, in principle, just compute the integral over \({\dsquare}\). 
Notwithstanding, as we know an exact expression for constant terms, there is no harm in including \(\rho(x) \mathtt{G}_\varepsilon(x)\) for extra accuracy.

\subsection{The spectral element method}\label{sec:NLCH_SEM}

The spectral element method (SEM) combines the flexibility of the finite element method (FEM) to study complicated geometries while also achieving high convergence rates that are only limited by the regularity of the solution and not by the underlying functional expansions.
To achieve this, spectral element methods employ a pseudospectral grid on each element and combine the elements to create a more complicated domain. 

The way different elements interact with each other determines the design of a SEM. Patching methods generalise collocation by dealing with the strong form of the PDE directly, imposing matching conditions (patching) between individual elements \cite{Komatitsch1998}; see also \cite[Chapter 22]{Boyd2001}, \cite[Chapter 24]{Martinsson2019}, and \cite{Bernardi2001} and the references therein. Galerkin SEM solves the weak form of the PDE system \cite{Patera1984}; see also \cite[Chapter 10]{Quarteroni2017} and \cite{Gopalakrishnan2008}. For certain classes of PDE systems, these two approaches are essentially equivalent \cite{Funaro1986,Young2019}. 
For additional reviews and references on SEM, we refer the reader to \cite{Deville2002,Karniadakis2005,Canuto2006,Canuto2007} and the more recent works \cite{Shen2011,Giraldo2020}.

We will use the package \texttt{MultiShape} \cite{Roden2022,Roden2022a}, which implements the SEM using a patch method built upon the pseudospectral \texttt{2DChebClass} package \cite{Goddard-2017}. This link allows for precomputing several quantities on individual geometric elements that make up a multishape\footnote{\texttt{MultiShape}, with capital letters, refers to the software package, while a lower-case \embh{multishape} refers to a complex domain of interest.}.
The complicated domain, i.e., multishape, is constructed by combining these elements and imposing continuity matching conditions at the intersection boundaries between them. The resulting package has shown remarkable effectiveness in solving complicated nonlinear and nonlocal problems \cite{Roden2023}. Detailed documentation of the package can be found in \cite[Chapter 6]{Roden2023a}.

\begin{wrapfigure}[8]{r}{0.2\textwidth} 
    \centering
    \vspace{-1.25em}
    \includegraphics[scale=0.75]{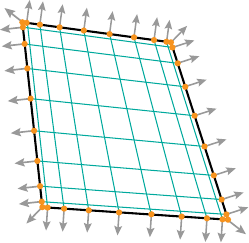}
    \caption{\texttt{Quadrilateral}.}
    \label{sub:MS_Quadrilateral}
\end{wrapfigure}
Among the shapes available in \texttt{MultiShape}, we will rely on the \texttt{Quadrilateral} class
to efficiently compute the Newtonian potential term \(\mathcal{I}_1\) in \cref{eq:Newtonian_Decomposition}. As its name suggests, this class allows us to build planar quadrilaterals that together compose a multishape; see \cref{sub:MS_Quadrilateral}. 
This tool allows us to compute the action of the convolution kernel in \( \Omega \setminus B_\infty (x; \varepsilon)  = \dsquare \) for different collocation points of \(x \in \mathcal{X}\). 
The action determines the rows of a matrix, which we denote by \( \Convh \). 

Throughout this study, we will assume that two-dimensional spectral meshes are generated by the same number of points in the horizontal and vertical axes, which we will label \(N\) and sometimes will refer to as the \embh{grid resolution}. 
We stress that this symmetry is introduced solely to avoid introducing separate resolution parameters for each axes; the theory, constructions, and computational results remain unchanged if the two directions are discretised with a different numbers of points.

\begin{center}
\begin{tcolorbox}[colback=algcolor, width=0.9\linewidth, colframe=white, boxrule=0pt, breakable=true]
\fontsize{9}{11}\selectfont
\begin{algorithm}[H]
\caption{Multishape method for singular kernels}
\label{alg:MultiShape}
\KwIn{%
	\(N \in \N\) (number of collocation points by dimension), 	\newline
	\(\varepsilon \in (0,\sfrac{1}{2})\) (neighbourhood radius), 
	\(\alpha \in \R_{>0}\) such that \(\alpha N \in \N\) (scaling factor).
}
\KwResult{\( \Convh \in \mathcal{M}_{N^2}(\R) \) approximating \cref{eq:Conv_Singular_Approximation_Scheme}. }
	Discretise \( \overline{\square} \) using a spectral mesh, yielding the collocation points \( \mathcal{X} = \{x_i\}_{i \in \llb 1,N^2 \rrb } \) \; 
	Create a floating box \( \boxdot  \gets \overline{B}_\infty(0;\varepsilon) \)		\;
\For{\(i\in \llb 1,N^2 \rrb \)}{
	%
	Compute the local box \( \scaleobj{0.75}{\blacksquare} \gets \overline{\square} \cap (x_i + \boxdot) \) 	\;
	Discretise \( \scaleobj{0.75}{\blacksquare} \) using \((\alpha N)^2\) collocation points 		\;
	Determine the holed box \( \dsquare \gets \overline{\square} \setminus (x_i + \boxdot)  \)  	\;
	Construct a multishape over  \(\dsquare\) with \( (\alpha N)^2 \)  collocation points per element		 	\;
	%
	%
	\( \Convh_i \gets \sbr{ \mathbb{I}_{\dsquare} \circ \del{ K(x_{1,i} - \dsquare, x_{2,i} - \dsquare)  } } \Pi_{\dsquare \to \square}  \)	 	\; 
	%
}
\end{algorithm}
\end{tcolorbox}
\end{center}

\cref{alg:MultiShape} proposes a procedure to iteratively build \( \Convh \in \mathcal{M}_{N^2}(\R)\). 
In a nutshell, a neighbourhood is traced around each collocation point, and a multishape is determined over its complement. We then approximate the matrix-row by interpolating the convolution matrix evaluated on the multishape to the original box.

The operator \( \Pi_{\dsquare \to \square}  \) interpolates the multishape discrete convolution to the baseline collocation points in \(\square\). As a result, we do not need to interpolate any quantity into the multishape to compute its partial convolution with \(K\). The operator \( \mathbb{I}_{\dsquare} \) consists of the Clenshaw--Curtis quadrature weights (see \cite[Equation 35]{Nold_2017}), extended to the multishape.

After evaluating \cref{alg:MultiShape}, we can compute the correction factor \cref{eq:G_Diagonal_Approximation_Newtonian} exactly from \cref{eq:CH_Conv_Neighbourhood_eps} and use a time stepping algorithm or other time-dependent PDE solver to solve \cref{PDS:NL-CH}.

The location of the collocation point and the size of the floating box \( \boxdot\) with respect to the base box \(\overline{\square}\) determines the geometry of \(\dsquare\).  We next present and test one method for computing the multishape over \(\dsquare\).

\subsubsection{Partition strategy}\label{ssec:Part_MS}

\begin{figure}[htbp]
\begin{center}
	\includegraphics[scale=0.8, page=2]{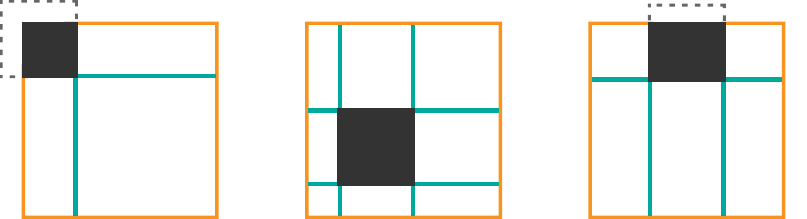}
	\caption{Maximal quadrilateral partitions based on 
	a neighbourhood around a collocation point.}
\label{fig:Multishape_Maximal}
\end{center}
\vspace{-0.5\baselineskip}
\end{figure}

We propose a partition of \(\dsquare\) based on the minimal number of rectangles required to determine a multishape. It turns out that this is the maximal number of quadrilaterals needed for such task that 
subdivide \(\dsquare\) into additional subregions whose sides extend exactly two sides of \( \scaleobj{0.75}{\blacksquare} \) to the boundary of \(\square\). We call this a \embh{maximal partition}\footnote{We could consider instead a partition based on the minimal number of quadrilaterals required to determine a multishape. This is the focus of \cref{App:MinPart_MS}, where we see that degenerate vertices limit the accuracy of the method.}.

We distinguish three partition methods based on the intersection with \( \scaleobj{0.75}{\blacksquare} \):
\begin{enumerate}[label=(\alph*), leftmargin=1.7em, itemsep=0.em]
\item
\( \scaleobj{0.75}{\blacksquare} \) is located at a corner of \(\overline{\square}\). Then, we only require \embh{3} elements to partition \(\dsquare\). These elements are determined by extending the sides of \( \scaleobj{0.75}{\blacksquare} \) inside \(\overline{\square}\) towards the border of the latter.
%
\item
\( \scaleobj{0.75}{\blacksquare} \) is completely embedded inside \(\square\) (i.e., \( \scaleobj{0.75}{\blacksquare} \Subset \square \)). Then we require \embh{8} elements determined by extending all sides of \( \scaleobj{0.75}{\blacksquare} \).
%
\item
\( \scaleobj{0.75}{\blacksquare} \) shares one side, and one side only, with \(\overline{\square}\). Then we need \embh{5} elements, which can be determined similarly as in case \textbf{(b)}.
\end{enumerate}

The procedure is depicted in \cref{fig:Multishape_Maximal}. There, a black-solid square represents \( \scaleobj{0.75}{\blacksquare} \) while the transversal segments inside \(\square\) (in teal) represent the additional lines required to determine a multishape in \(\dsquare\). 
Observe that if \( (\alpha N)^2\) is the number of collocation points used to discretise \( \scaleobj{0.75}{\blacksquare} \), then we will deal with at most \( 8 (\alpha N)^2 \) points at once for most iterations within the inner loop of \cref{alg:MultiShape}. 


\subsubsection{Quality of approximation against constant functions}
\label{ssec:Quality_test}

We previously derived in \cref{sec:CH_Newtonian_Exact_Formulae} a closed-form expression for the convolution of the Newtonian potential against the constant function \( \mathsf{1}_\square\), which we denoted \(J\); see \cref{eq:CH_Conv_Whole_Domain}. 
We now present a set of numerical tests employing \cref{alg:MultiShape} to test the validity of the approximation \cref{eq:Conv_Singular_Approximation_Scheme} and assess its accuracy against its analytical counterpart. For this purpose, we use the pointwise absolute error
\begin{equation}
\label{eq:MS_error_NP_conv_abs}
	e_{\varepsilon} (x) = {  \big| \mathtt{J}(x) - \mathcal{I}_1[1_\square](x) - \mathtt{G}_\varepsilon(x) \big|  }.
\end{equation}
We do not report the relative pointwise error, since \( \mathtt{J}(x) \) vanishes at certain points of the domain. In such regions, division by small values would artificially inflate the relative error and obscure the actual accuracy of the approximation.

First, we evaluated \cref{eq:MS_error_NP_conv_abs} for \(\varepsilon \in \{10^{-2},10^{-5}\} \).
We selected comparable scaling factors, thus for \( N=10\) we have \( \alpha \in \{2,10,16\}\), while for \( N=20\) we have \( \alpha \in \{1,5,8\}\), thus the same number of points per element is considered when running \cref{alg:MultiShape}.



%
%

\begin{figure}[h!]
\centering
\setlength{\unitlength}{1cm}  
    \subcaptionbox{Swarm plots for scalings of \(N=10\).}[7.9cm]{
    \begin{picture}(8,6.6)   
	\put(0,0){
	    \includegraphics[scale=0.98]{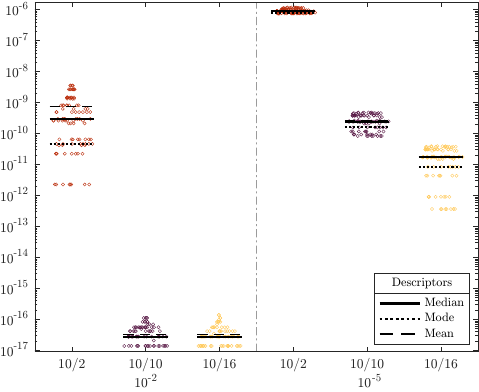}
	}
	\put(4.2,-0.3){\makebox(0,0)[c]{\fontsize{6.5}{7.5}\selectfont Instance configuration for $\alpha   \in \{2,10,16\}$ }}
	\put(-0.35,3.){\rotatebox{90}{ \fontsize{6.5}{7.5}\selectfont Error $e_{\varepsilon}$} }
	\end{picture}
	\vspace{0.25em}
	}
    \hspace{2em}
    \subcaptionbox{Swarm plots for scalings of \(N=20\).}[7.9cm]{
    \begin{picture}(8,6.6)   
	\put(0,0){
	    \includegraphics[scale=0.98]{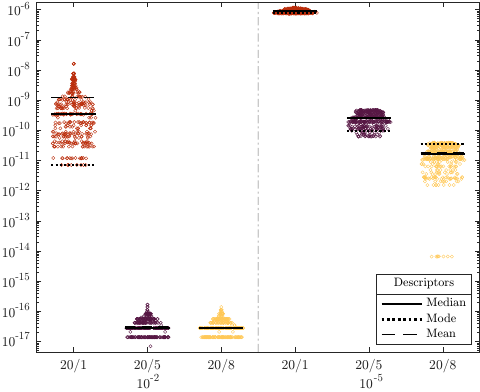}
	}
	\put(4.2,-0.3){\makebox(0,0)[c]{\fontsize{6.5}{7.5}\selectfont Instance configuration for $\alpha   \in \{1,5,8\}$ }}
	\put(-0.35,3.){\rotatebox{90}{ \fontsize{6.5}{7.5}\selectfont Error $e_{\varepsilon}$} }
	\end{picture}
	\vspace{0.25em}
	}
	%
	%
    \caption{
	Swarm plots of the relative errors \(e_{\varepsilon}\) for approximating the Newtonian potential under refinement for the maximal partition strategy. The horizontal labels are presented in the format \( N/\alpha\), corresponding to the grid resolution 
	$N \in \{10,20\}$ and the scaling factor $\alpha$. Additionally, two options for the neighbourhood radius are presented, specifically $\varepsilon \in \{ \num{e-2}, \num{e-5} \}$.
    }
\label{fig:K_Test_MaxQ_NP_alpha_10}
\vspace{-0.5\baselineskip}
\end{figure}
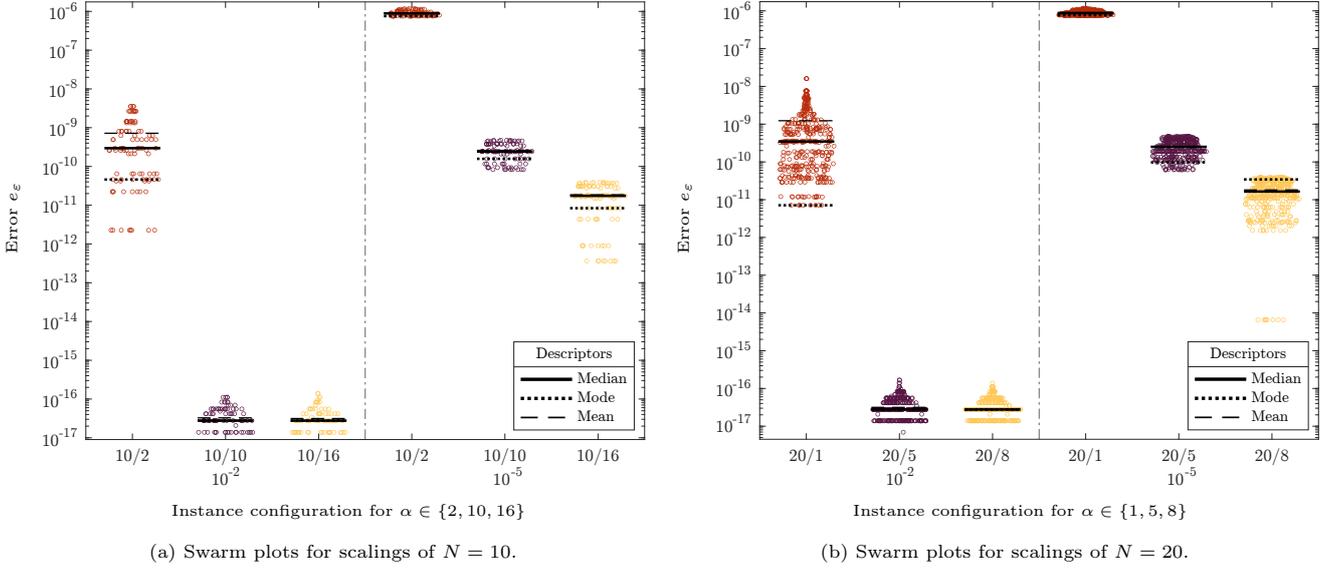


The results are presented in \cref{fig:K_Test_MaxQ_NP_alpha_10}. We observe that the refinement strategy improves the convergence of the method, and that, in essence, both choices of \(N\) generate the same error distribution. In particular, the errors are mostly clustered around the mean (i.e., they vary in a small range). Furthermore, we observe that spectral convergence is attained for two cases of \(\varepsilon = 10^{-2}\). The case \( \varepsilon = 10^{-5}\) suggests that further refinement can achieve the same level of accuracy. However, the attained error values are already below common thresholds used for solving differential systems numerically, and thus suggest that the approximation is sufficient for solving \cref{PDS:NL-CH}.

As a second numerical test, let us fix \(N=20\) and study the maximum error as \(\varepsilon\) decreases. 
\cref{fig:Multishape_Maximal_Q_MaxError} 
contains two panels that display this evolution for five different scaling factors \( \alpha \in \{1,2,3,4,8\}\). 
Panel \textbf{(a)} displays the maximum value of \(e_{\varepsilon}\).
If we write the scaling factor as \(\alpha = 2^k\) for some \(k \in \N\) and fix the value of \(\varepsilon\), then increasing \(k\) by one unit 
results in a reduction of \(e_{\varepsilon}\) of about one order of magnitude.
The dashed line represents the maximum value of \(\mathtt{G}_\varepsilon \), as computed in \cref{eq:CH_Conv_Neighbourhood_eps_min}. 
The results of an additional computation are presented in panel \textbf{(b)}. Here, we present the error of computing the convolution just inside the multishape; i.e., we report the maximum value of 
\(
	\big| \mathtt{J}(x) - \mathcal{I}_1[1](x) \big|  
\).
In this case, it is clear that the approximated integral inside the multishape  quickly approaches \(\mathtt{G}_\varepsilon \) for moderate values of \(\varepsilon\). For smaller values, we can increase the scaling factor and thus reduce the impact of the correction term. 
Taking into account the fact that the maximum of \(\mathtt{G}_\varepsilon \) decreases as \( O(\varepsilon^2 \log\varepsilon) \), we can thus select a value of \(\varepsilon\) and a scaling factor \(\alpha\) that significatively reduces the error term while also keeping the size of \( \scaleobj{0.75}{\blacksquare} \) as small as possible.

Let us finalise this section with some additional comments. As we noted at the beginning of this section, the cost of computing the convolution matrix \(\Convh\) increases as the scaling factor grows. For the case \( N= 20\) and \( \alpha = 8\), this yields a multishape with around \( 2 \times 10^{5}\) degrees of freedom for the case \( \scaleobj{0.75}{\blacksquare} \Subset \square \). Notwithstanding, our implementation is able to perform the kernel approximation for this point in a matter of seconds using a personal computer. 
Moreover, the resulting size of the convolution matrix is independent of any of the choices of discretisation; it is just the size of the multishape. As a result, in cases where the kernel is constant in time (as is the case for the Cahn--Hilliard system), the precomputation of \( \Convh \) is only done once and does not affect the computational cost of, e.g., solving the resulting PDE. For our specific setting, we will test the precomputed matrices associated with the Newtonian potential and the configuration parameters \( N\in \{10,20,40\} \) and \(\varepsilon \in \{10^{-2}, 10^{-5} \}  \). On the other hand, our results suggest that we could aim to build a more efficient partition of \( \dsquare\) that still showcases a significant error reduction. 
This is left as an avenue for future research.

%
%
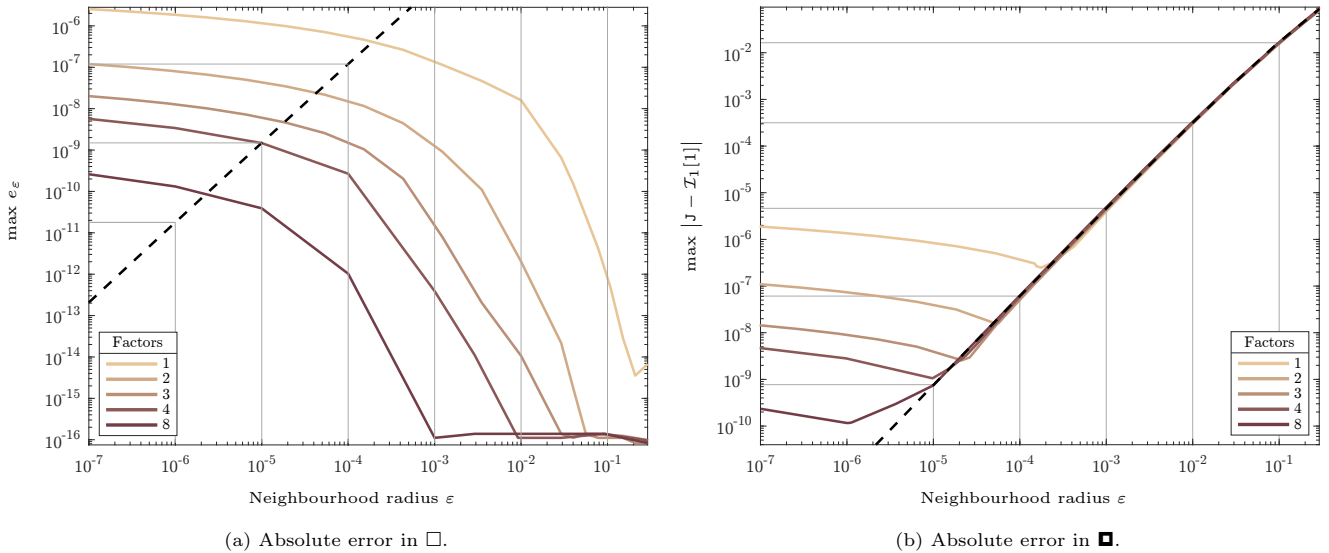
\begin{figure}[h!]
\centering
\vspace{1.75cm}
\setlength{\unitlength}{1cm}  
    \subcaptionbox{Absolute error in \(\square\).}[7.9cm]{
    \begin{picture}(8,4.8)  
	\put(0,0){
	    \includegraphics[scale=0.98]{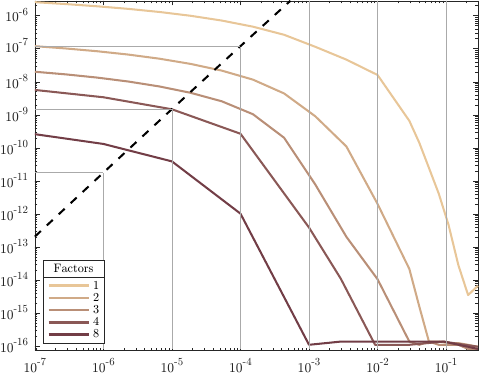}
	}
	\put(4.2,-0.3){\makebox(0,0)[c]{\fontsize{6.5}{7.5}\selectfont Neighbourhood radius $\varepsilon$ }}
	\put(-0.35,3.0){\rotatebox{90}{ \fontsize{6.5}{7.5}\selectfont $\max \, e_{\varepsilon}$} }
	\end{picture}
	\vspace{0.25em}
	}
    \hspace{2em}
    \subcaptionbox{Absolute error in \(\dsquare\).}[7.9cm]{
    \begin{picture}(8,4.8)  
	\put(0,0){
	    \includegraphics[scale=0.98]{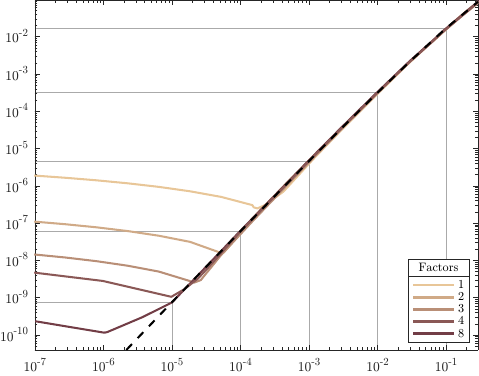}
	}
	\put(4.2,-0.3){\makebox(0,0)[c]{\fontsize{6.5}{7.5}\selectfont Neighbourhood radius $\varepsilon$ }}
	\put(-0.35,2.6){\rotatebox{90}{ \fontsize{6.5}{7.5}\selectfont $\max \, \big| \mathtt{J} - \mathcal{I}_1[1] \big|  $} }
	\end{picture}
	\vspace{0.25em}
	}
	%
	%
    \caption{
	Evolution (in log-log scale) of the absolute error (vertical axes) as \(\varepsilon\) varies in \( [10^{-7}, \sfrac{3}{10}]\) (horizontal axes). The dashed line is given by the maximum of \( |\mathtt{G}_\varepsilon|\); see \cref{eq:CH_Conv_Neighbourhood_eps_min}.
    }
\label{fig:Multishape_Maximal_Q_MaxError}
\vspace{-1\baselineskip}
\end{figure}


\section{Numerical results}\label{ch:NLCH_Numerical_Experiments}

In this section, we present a collection of numerical experiments designed to validate the performance of the pseudospectral multishape method introduced in \cref{ch:Spectral_Element_Newtonian_Potential} when applied to the nonlocal Cahn--Hilliard system \cref{PDS:NL-CH}. Rather than aiming at an exhaustive numerical analysis%
\footnote{%
Additional tests studying exact solutions and mesh refinement can be found in \cite[Chapter 20]{AMT2025a}.
},
our objective is to illustrate the predicted qualitative (and quantitative) behaviour of solutions under representative modelling scenarios.

Since the existence and uniqueness theory from \cref{known} does not depend on the sign of the interaction kernel, we consider the family of scaled kernels \( K_\eta \coloneqq \eta K \), where \( \eta \in \R\) is a constant that allows us to explore different dynamical regimes by tuning a single parameter. Accordingly, we focus on numerical solutions of the system
\begin{equation}
	\label{sys:NL-CH_Particular}
	\left\{
	\begin{aligned}
		\dod{\rho}{t} &= \Delta \mu						&&\text{in } Q,
		\\
		\mu &= F'(\rho) - K_\eta \star \rho		\qquad		&&\text{in } Q,
		\\
		\partial_n \mu &= 0							&&\text{on } \partial\Omega \times (0,T),
		\\
		\rho(0) &= \rho_0							&&\text{in } \Omega.
	\end{aligned}
	\right.
\end{equation}

Throughout all experiments, we assume a constant mobility \(m \equiv 1\) and employ the logarithmic potential \cref{eq:LogPotential} with scaling parameter \(\theta = 2\), for which
\( 
	F'(s) = \log \frac{1+s}{1-s}
\).
In addition, we compute the associated nonlocal free energy
\begin{equation}
	\label{eq:Nonlocal_Energy_NP}
	\mathcal{E}_\eta(t)
	\coloneqq \mathcal{F}[\rho](t)
	= \into F(\rho) \dif x - \frac{1}{2} \into \del{K_\eta \star \rho} \rho(x) \dif x
	\qquad (\forall t > 0),
\end{equation}
which is expected to be non-increasing along trajectories of the system.

We also monitor several additional quantities of interest during the simulations. The existence and qualitative behaviour of these quantities follow from the analytical results established in \cref{ch:CH_WP_and_Setup} under assumptions on \(F\) and its derivatives that are satisfied by the logarithmic potential and some of its variants. In particular, we are interested in parameter regimes (i.e., initial conditions and values of \(\eta\)) for which the \embh{separation property} holds. As established in \cref{sep1}, solutions to the nonlocal Cahn--Hilliard system become instantaneously and uniformly separated from the pure phases \(\pm1\), in the sense that there exists \(\delta>0\) such that for large-enough \(t\) the \(L^\infty\) norm of a solution \( \rho\) is bounded above by \( 1-\delta\). 
Qualitatively, this implies that an initially mixed phase \(\rho_0\) may segregate into distinct, nearly homogeneous regions in finite time, dissipate via diffusion, or evolve towards a stationary mixed state.
Furthermore, the long-time behaviour of the system is characterised by convergence to a stationary state. In particular, \cref{equi} guarantees the existence of a stationary solution \(\rho_\infty\) and a constant \(\mu_\infty\) such that \( \mu_\infty = F'(\rho_\infty)  - K \star \rho_\infty \) and the average estimate \( \overline{\rho}_\infty = \overline{\rho}_0\) holds. Finally, \cref{Conv_Rate} asserts that solutions converge to \(\rho_\infty\) in the \(L^2(\Omega)\) norm at a polynomial rate in time.

In what follows, we present three sets of experiments. In \cref{sec:NLCHnum_Newtonian_Potential}, we solve the nonlocal Cahn--Hilliard system with logarithmic potential, for different scalings of the Newtonian potential that showcase interesting dynamics for two different initial conditions. Then, in \cref{sec:NLCHnum_N_Regular_Potential}, we present a regularised version of the logarithmic potential, which can be used to approximate the solution of the original system for large times. Finally, \cref{sec:NLCHnum_Mixed_Kernel} explores a linear combination of the Newtonian potential with a mollifier (an approximation of unity) to obtain a discretised convolution kernel that displays a wide range of, entry-wise, positive and negative values. The resulting dynamics induced by this kernel are showcased for several values of \(\eta\). Additional tests on the unit disc and the unit cube are presented in \cref{App:Disc,App:Cube}.

All the experiments were performed on a MacBook Pro 2020 M1 with 16 GB RAM. 
Unless otherwise stated, we employed the variable-step, variable-order (1 to 5) stiff solver \texttt{ode15s} from \textsc{Matlab} 2025b. The method is based on backward differentiation and a modified Rosenbrock method \cite{Shampine1997} and includes an optional differential-algebraic equation (DAE) solver, which we use to enforce boundary conditions \cite{Shampine1999}. We set the method's absolute and relative tolerances to \num{e-7}.
Additionally, we used the configuration \( (N,\alpha,\varepsilon) = (40, 4, 10^{-5}) \) for approximating the Newtonian potential. 

The code, the associated kernels, energy and convergence plots, and additional experiments throughout scales can be found in  the open source repository:
\begin{center}
	\href{https://github.com/DDFT-Modelling/NLCH}{\texttt{https://github.com/DDFT-Modelling/NLCH}}
\end{center}


\subsection{Scalings of the Newtonian potential}\label{sec:NLCHnum_Newtonian_Potential}

In this subsection, we investigate how the scaling of the Newtonian interaction kernel influences the qualitative behaviour of solutions to the nonlocal Cahn--Hilliard system. In particular, we illustrate how varying the parameter \(\eta\) modulates the competition between diffusion and nonlocal attraction, leading to markedly different dynamical regimes, including pure diffusion, phase separation, and convergence to nontrivial stationary states.

To this end, we consider numerical simulations based on the scaled Newtonian kernel
\[
	K_\eta : \R^2 \setminus\{0\} \ni x \longmapsto \frac{\eta}{2\pi}\log\|x\| \in \R,
\]
where \(\eta \in \R\) is treated as a tunable parameter.

\subsubsection{Periodic wave}

\begin{figure}[ht]
\begin{center}
    \subcaptionbox{
    	Diffusion for \( \eta = 1\).		\label{sfig:A_Solution_PS_a}
	\vspace{0.5em}
    }
    {
    \includegraphics[scale=0.7]{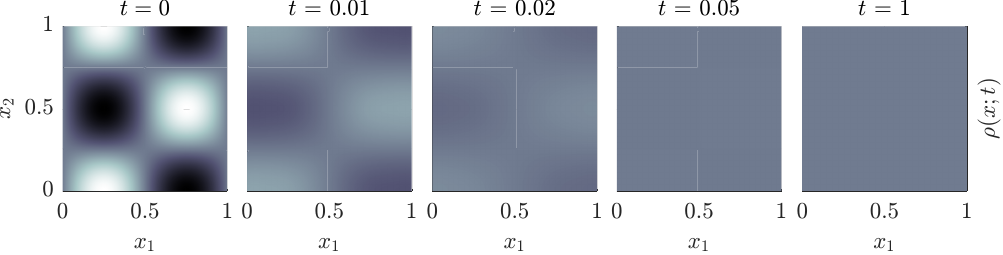}
    }
    \\[0.5\baselineskip]
    \subcaptionbox{
    	Slow phase separation for \( \eta = -50\).		\label{sfig:A_Solution_PS_b}
	\vspace{0.5em}
    }
    {
    \includegraphics[scale=0.7]{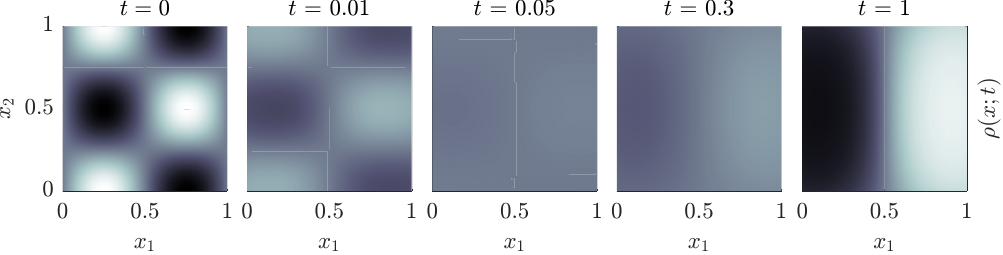}
    }
    \\[0.5\baselineskip]
    \subcaptionbox{
    	Fast phase separation for \( \eta = -150\).		\label{sfig:A_Solution_PS_d}
	\vspace{0.5em}
    }
    {
    \includegraphics[scale=0.7]{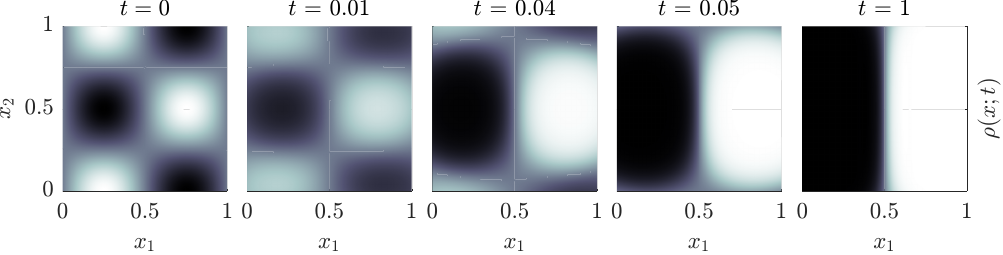}
    }
\caption{Evolution of \(\rho(x;t)\) from the initial condition \cref{eq:NLCH_Initial_Cond_SinCos} and different values of \( \eta\).}
\label{fig:A_Solution_PS}
\end{center}
\vspace{-0.5em}
\end{figure}

We obtained solutions of \cref{sys:NL-CH_Particular} up to \(T = 1\) for trajectories induced by the initial condition given by the periodic wave
\begin{equation}
	\label{eq:NLCH_Initial_Cond_SinCos}
	\rho_0(x) =  \sin(2\pi x_1) \cos(2\pi x_2),
	\qquad
	\overline{\rho}_0 = 0
	.
\end{equation}
We tested \( \eta\) taking values inside the interval \( [-150, 150] \). 
Whenever the solver struggled with the initial condition, we instead employed a fixed-point method to solve the boundary data \cite[\S 20.3.2]{AMT_Th_2024}.

In all cases, we observed that the mass of the system increased slightly to a value bounded by \( 10^{-10}\), which is of the same order of magnitude as the error from the convolution approximation, recall \cref{eq:CH_Conv_Neighbourhood_eps_min}, and smaller than the error tolerance of the DAE solver.
Moreover, we observed that for values of \(\eta\) greater than \(-30\), the convolution term was dominated by \(F'\) and \( \rho_0 \) was diffused to the stationary state \( \rho_\infty \equiv 0\). 
Here, we highlight that the case \( \eta = 0\) resulted in a stiff system, yet the same equilibrium was attained.
For values of \(\eta\) lower than \(-40\), the convolution was able to compete against the values of \(F'\) and other stationary points were attained. 
As \(\eta\) became increasingly negative, the system started to approximate the step function \( 2 \del{ \mathsf{1}_{ \R_{\geq 0} } (x_1 - \sfrac{1}{2} ) } - 1 \).
Note that this function does not depend on \( x_2\). 
Overall, we found that the system evolves faster towards a stationary state if we increase the absolute value of \(\eta\).

In \cref{fig:A_Solution_PS} the evolution of \(\rho\) is portrayed for the scaling factors \( \eta \in \{1,-50,-150\}\) (one value per panel). 
Here, values close to \(-1\) correspond to near-black shades, while values close to 1 appear as the brightest areas. Intermediate grayscale tones represent values around zero. Each panel displays a different set of time points. This is done to ensure that significant phenomena are portrayed.
Panel \textbf{(a)} showcases an example of diffusion, where the initial condition quickly vanishes towards the stationary state \( \rho_\infty \equiv 0\). 
In contrast, panels \textbf{(b)} and \textbf{(c)} display the phase separation property such that, already by the moment \(t\) takes the value of \(1\), an equilibrium that resembles the step function \(p\) is reached. 
More precisely, in panel \textbf{(b)}, we observe that the system first diffuses the initial state and by time \( t = 0.05\) reaches an almost-constant state. Then, the solution undergoes a separation phase, which is seen in the intensity values corresponding to \( t =0.3\) and \( t = 1\).
Panel \textbf{(c)} displays a similar behaviour that attains a more distinguishable separation at a faster rate.


\begin{figure}[htbp]
\begin{center}
	\subcaptionbox{Evolution of the nonlocal free energy \cref{eq:Nonlocal_Energy_NP}. \label{fig:A_NLCH_Energy_PS} }
	{\includegraphics[scale=0.45]{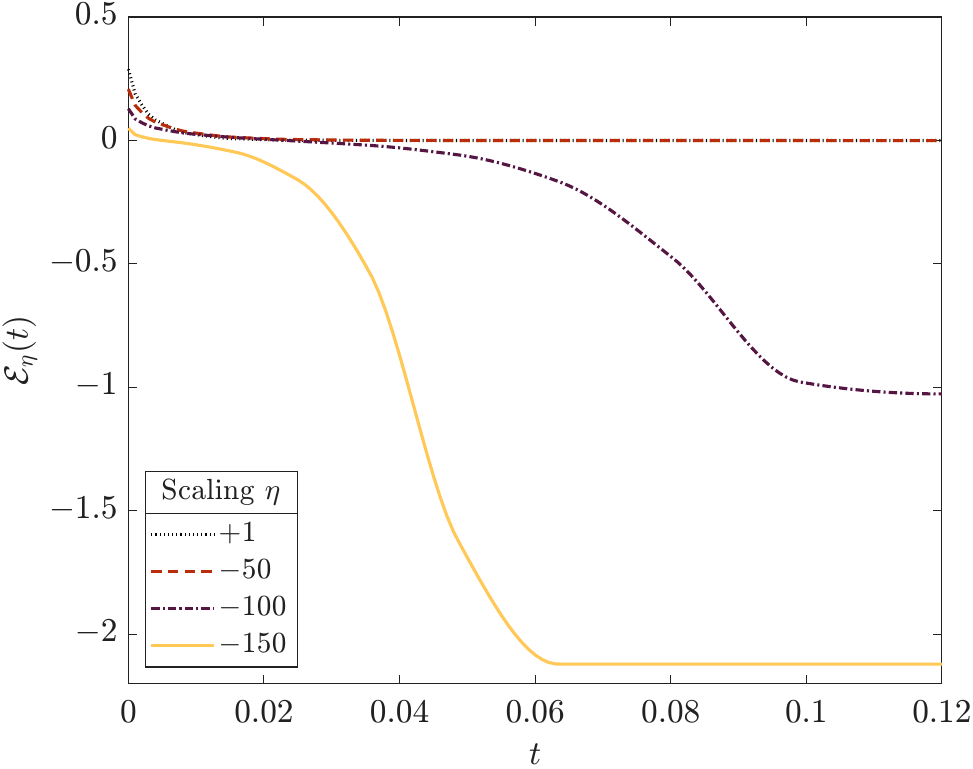}}
	\quad
	\subcaptionbox{Convergence towards equilibrium in the $L^2$ norm. \label{fig:A_NLCH_Convergence_PS} }
	{\includegraphics[scale=0.45]{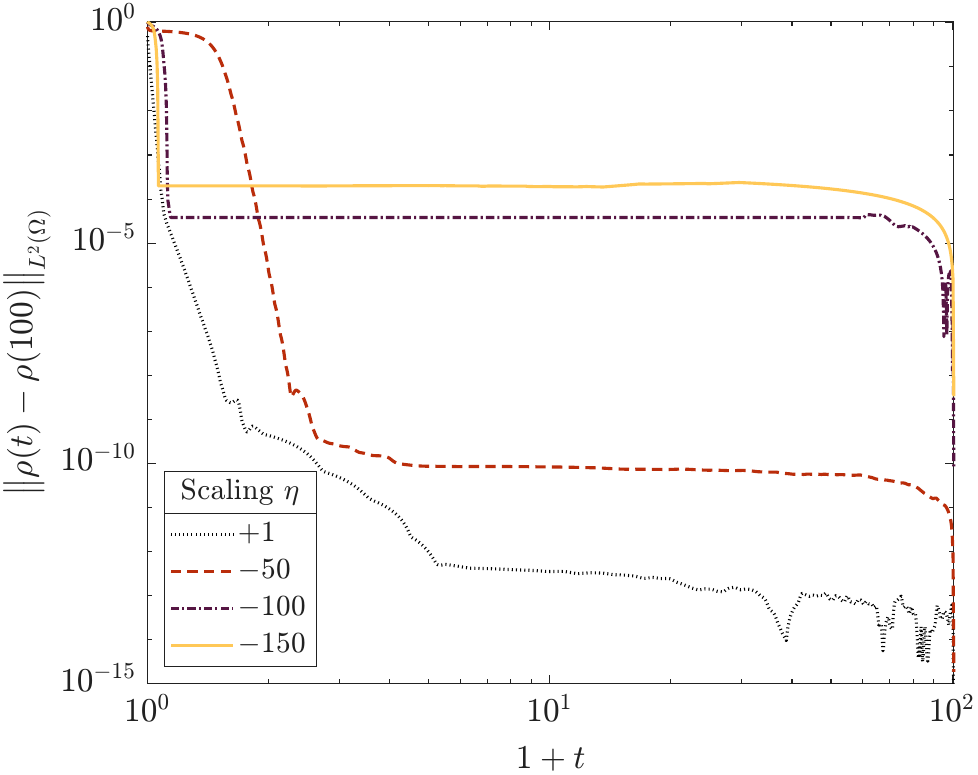}}
\caption{Free energy dissipation and convergence towards equilibrium for different values of the scaling parameter \(\eta\).}
\label{fig:A_NLCH_EnConv_PS}
\end{center}
\vspace{-0.5\baselineskip}
\end{figure}

The evolution of the nonlocal free energy \cref{eq:Nonlocal_Energy_NP} and the convergence towards equilibrium for this set of examples are reported in \cref{fig:A_NLCH_EnConv_PS}. 
In panel \cref{fig:A_NLCH_Energy_PS}, we display the free energy evolution over the time interval \(t \in [0,0.12]\) only, as all simulations rapidly reach an equilibrium that remains essentially unchanged thereafter. This restricted time window is therefore chosen solely for graphical clarity. As expected, the free energy is non-increasing in time. Moreover, we observe that the equilibrium free  energy decreases as the scaling parameter $\eta$ decreases. In particular, values of $\eta$ that induce phase separation lead to negative free energies, whereas the diffusive regime converging to the homogeneous state is characterised by a free energy bounded below by zero. Note that the free energy is only unique up to a constant in the sense that we can shift it without any repercussion in the dynamics.

The corresponding convergence behaviour is shown in panel \cref{fig:A_NLCH_Convergence_PS}. Here, we repeated our tests for \(T =100\), for which \( \rho_\infty \approx \rho(100)\). The error with respect to the final state is reported in the \(L^2\) norm. All cases exhibit an initial transient phase, after which the relaxation towards equilibrium slows down and enters a gradual decay regime. This behaviour is consistent with the polynomial convergence predicted by \cref{Conv_Rate}. We note that for large negative values of the scaling parameter, the convergence curves exhibit extended intervals during which the error remains nearly constant. This behaviour does not indicate a loss of convergence, but rather reflects the fact that, after a rapid initial separation, the solution reaches a near-stationary configuration. In this regime, the subsequent evolution is governed by very small corrections (the specialisation phase), for which any additional dynamics fall below the numerical resolution.

\begin{table}[h!]
\setlength{\tabcolsep}{0.4em}
\centering
\fontsize{8}{9}\selectfont
\def\arraystretch{1.5}
\begin{tabular}{c c rrrrrrrr}
\toprule
\( \eta \) && \multicolumn{1}{c}{150} & \multicolumn{1}{c}{50} & \multicolumn{1}{c}{1} & \multicolumn{1}{c}{0} & \multicolumn{1}{c}{$-1$} & \multicolumn{1}{c}{$-50$} & \multicolumn{1}{c}{$-100$} & \multicolumn{1}{c}{$-150$}
\\
\midrule
	\( \delta \)  
	&&
	\(1- \num{2e-9}\) & \(1 - \num{2e-10}\) & \(1- \num{3e-11}\)	& 
	\( 1 - \num{2e-9}\) & \( 1-\num{2e-12}\)  
	& 0.13\hphantom{00.}  & \(\num{2e-3}\)  & \(\num{6e-5}\)
	\\
	\( \mu_\infty \)  && 
	\num{9e-9} & \num{8e-10} &\(\num{6e-11}\) &  
	\(-\num{4e-9}\) &  \num{4e-12} 
	&  \num{3e-11} &  \num{e-10} &  \num{3e-10}
	\\
\bottomrule
\end{tabular}
\\[0.5em]
\caption{Approximation of the values \(\delta\) and \(\mu_\infty\) for different choices of \(\eta\).}
\vspace{-0.5\baselineskip}
\label{tb:A_CH_PS_param}
\end{table}


Finally, we computed the approximate values of \(\delta\) and \(\mu_\infty\), displayed in \cref{tb:A_CH_PS_param} for eight values of \(\eta\) (guaranteed by \cref{equi,sep1}). Once more, we used the extended timeline \(T=100\). Nevertheless, we highlight that by \(t\approx 1\) the solution has already entered a numerically stationary regime for all tested values of \(\eta\), as indicated by the free energy profile \cref{fig:A_NLCH_Energy_PS}. We notice that the values of \(\delta\) that correspond to solutions behaving as the stationary point \( \rho_\infty \equiv 0\) display a value of \(\delta\) close to the error of the convolution term. Then, as the phase separation kicks in, the values of \(\delta\) decrease as \(\rho\) starts to approach the separated state \(\rho_\infty\). To approximate the value of \(\mu_\infty\) we took the total average of \( \bfv =  F'(\rhob(T)) - \mathbb{K} \star \rhob(T)\), where \(\rhob\) is the numerical approximation of \(\rho\) as a function of time. We also noticed that \( \| \bfv - \overline{\bfv} \|_\infty \leq 10^{-13} \) and in most cases reached machine precision; i.e., in most cases \( \bfv \) behaves like a constant.

\subsubsection{A function of compact support}

Let us define the function
\(
	q(x) \coloneqq \del{ \sfrac{1}{2} \sin(3 \pi x_1) \cos(3 \pi x_2)  + \sfrac{1}{4}  } \, \,  \mathsf{1}_{ B_\infty (\sfrac{1}{2} \Ones[2] ; \sfrac{9}{25} ) } (x) .
\)
This function has compact support in \(\Omega\), yet it is not differentiable. We will regularise \(q\) using the mollifier \cite[Chapter 4]{Brezis_2011}:
\begin{equation}
	\label{eq:Mollifier}
	H(x) \coloneqq \exp\del{ - \del{1-\|x\|_2^2}^{-1} }
	\, \,  \mathsf{1}_{ B_2 (0;1) } (x) .
\end{equation}
We then select the initial condition
\begin{equation}
	\label{eq:NLCH_Initial_Cond_Conv}
	\rho_0(x) = 3 H_a \star q(x),
\end{equation}
where \(H_a (x) \coloneqq a^{-2} H( \sfrac{x}{a} ) \) for some \(a > 0\). In particular, let us select \( a = 10^{-1}\) which guarantees that \( \rho_0 \big|_{\partial \Omega} = 0\) and \( |\rho_0| < 1\).
Computationally, we also observe that \( \overline{\rho}_0 \approx 0.18 \).

\begin{figure}[ht]
\begin{center}
    \subcaptionbox{
    	Diffusion for \( \eta = 500\).		\label{sfig:C_Solution_PS_a}
	\vspace{0.5em}
    }
    {
    \includegraphics[scale=0.7]{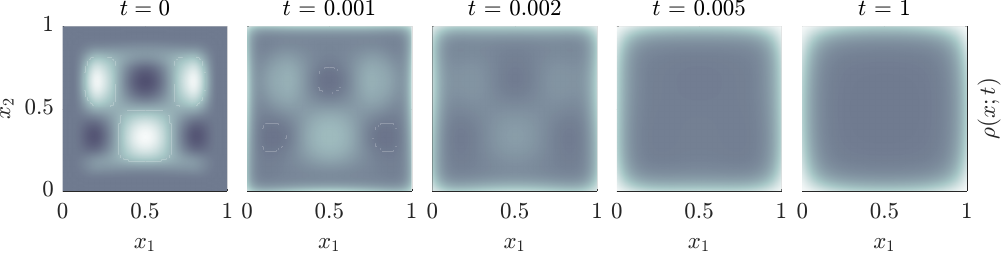}
    }
    \\[0.5\baselineskip]
    \subcaptionbox{
    	Slow phase separation for \( \eta = -100\).		\label{sfig:C_Solution_PS_b}
	\vspace{0.5em}
    }
    {
    \includegraphics[scale=0.7]{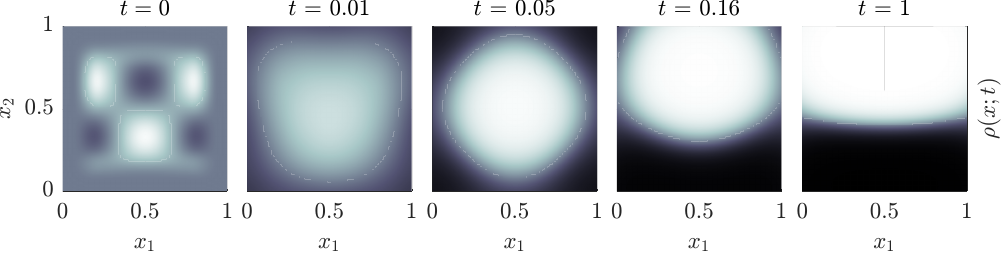}
    }
\caption{Evolution of \(\rho(x;t)\) from the initial condition \cref{eq:NLCH_Initial_Cond_Conv} and two choices of \( \eta\).}
\label{fig:C_Solution_PS}
\end{center}
\vspace{-0.5\baselineskip}
\end{figure}

Running the system for \( T = 1\) and \( \eta \in [-100,500] \), we obtained a similar behaviour to what we observed in the previous two examples: Energies increase as \(\eta\) increases, there is a plateau for \(\delta\), mass is conserved, and stationary states are reached in finite time. 
\cref{fig:C_Solution_PS} showcases two particular cases that reach states with the separation property. Panel \textbf{(a)} displays a state that concentrates mass around the boundary of \(\Omega\). In contrast, in panel \textbf{(b)} we see that the state undergoes two transitions: mass accumulates inside the unit ball, and then the system approximates the step function \( 2 \del{\mathsf{1}_{ \R_{\geq 0} } (x_2 - \sfrac{9}{20} )} - 1 \), which is independent of \( x_1\).

\subsection{A regular potential}\label{sec:NLCHnum_N_Regular_Potential}

We introduce a regularised potential \(F_\omega\in C^3(\mathbb{R})\) which, for any \(\omega\in(0,1)\), extends the logarithmic potential \(F\) outside the interval \([-1+\omega,\,1-\omega]\) by matching its third-order Taylor expansions at the cutoff points:
\begin{equation*}
	F_{\omega} (s) \coloneqq 
	\begin{cases}
		\sum\limits_{k=0}^3 \frac{1}{k!} F^{(k)} (-1+\omega)  (s+1-\omega)^k	
			& 	\text{if } s \leq -1 + \omega,
		\\
		F(s) 	& 	\text{if } s \in (-1+\omega, 1-\omega),
		\\
		\sum\limits_{k=0}^3 \frac{1}{k!} F^{(k)} (1-\omega)  (s-1+\omega)^k		
			& 	\text{if } s \geq 1 - \omega.
	\end{cases}
\end{equation*}
Recalling that \( F^{(3)} (s) = \sfrac{ 4s }{(s^2-1)^2} \), we obtain that
\begin{align*}
	F_{\omega} '(s) =
	\begin{cases}
		\log \frac{\omega}{2-\omega} + \frac{2}{\omega(2-\omega)} (s+1-\omega) - \frac{2(1-\omega)}{\omega^2 (2-\omega)^2}  (s+1-\omega)^2
			& 	\text{if } s \leq -1 + \omega,
		\\
		\log \frac{1+s}{1-s} 	& 	\text{if } s \in (-1+\omega, 1-\omega),
		\\
		\log \frac{2-\omega}{\omega} + \frac{2}{\omega(2-\omega)} (s-1+\omega) + \frac{2(1-\omega)}{\omega^2 (2-\omega)^2} (s-1+\omega)^2
			& 	\text{if } s \geq 1 - \omega.
	\end{cases}
\end{align*}
This regularisation of \(F\) was introduced and used in \cite{Gal2017} to show how the original strictly separated solution can be viewed as a solution to the same equation with a smooth potential. In particular, it is shown that if \(\rho\) solves the original system and \(\rho_\omega\) solves the corresponding regularised system with initial condition \(\rho_\omega(x,0) = \rho(x,3\sigma)\), then \(\rho_\omega(x,t) = \rho(x,t+3\sigma)\) for any \(\sigma>0\). The constant factor multiplying \(\sigma\) allows one to select \(\omega\) as small as the separation parameter \(\delta\); thus, it does not affect the results of this section, yet we will preserve it for consistency with the theory.

We now investigate this property numerically. We consider a solution \(\rho\) of \cref{sys:NL-CH_Particular} with initial condition \cref{eq:NLCH_Initial_Cond_Conv} and scaling parameter \(\eta=500\). The values of \(\sigma\) and the integration time \(T\) are chosen so that the simulation does not start from a stationary state; see \cref{fig:D_Solution_PS}. Specifically, we set \( \sigma = (\sfrac{2}{9}) \times 10^{-3} \)  and \(T = \sfrac{1}{375} \), which yield \( T + 3\sigma = (\sfrac{1}{3}) \times 10^{-2} \), and fix \( \omega = 10^{-3} \).

\begin{figure}[ht]
\begin{center}
    {
    \includegraphics[scale=0.7]{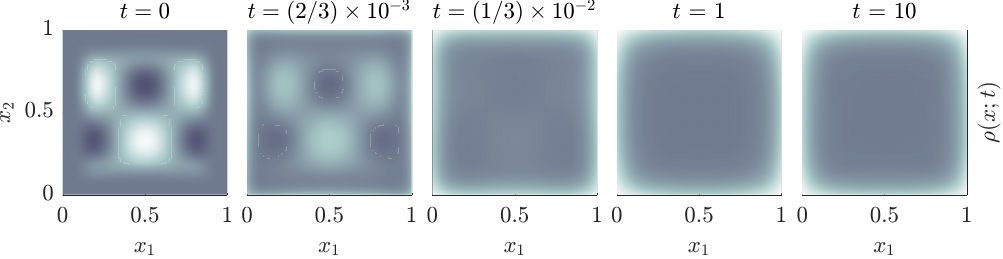}
    \vspace{-0.25\baselineskip}
    }
\caption{Evolution of \(\rho(x;t)\) from the initial condition \cref{eq:NLCH_Initial_Cond_Conv} at times \( 3\sigma\), \( T+3\sigma\), \(1\), and \(10\).}
\label{fig:D_Solution_PS}
\end{center}
\vspace{-0.5\baselineskip}
\end{figure}

The numerical values of \(\rho(x,3\sigma)\) and \(\rho(x,T+3\sigma)\) were computed using the DAE solver in \textsc{Matlab} with absolute and relative tolerances set to \num{e-11}. We then set \(\rho(x,3\sigma)\) as initial condition for the regularised problem and computed \(\rho_\omega(x,T)\). The numerical \(L^2(\Omega)\) error,
\(
	\big\| \,  \rhob_\omega (T) - \rhob( T + 3\sigma)  \big\|_{L^2(\Omega)} 
\)
was approximately \num{2e-11}.
As expected, the agreement improves once the solution approaches a stationary state. Repeating the experiment with integration up to time \(T=1-3\sigma\) yields an approximate error of \num{8e-15}.
We have thus obtained a good approximation that matches (to machine precision) the stationary state of \(\rho\) itself.
Motivated by this observation, we next assess whether the regularised problem can be used to accelerate long-time simulations. Reducing the solver tolerances to \num{e-7} and integrating up to time \(T=10-3\sigma\), we obtain an approximation error of around \num{9e-14},
indicating that the stationary state is still recovered to high accuracy at significantly reduced computational cost.

The reduction in computational cost is expected from the improved regularity of the potential. The blow up of \(F'\) near pure states \( (|\rho| \sim 1\)) adds stiffness into the semi-discrete system, which restricts admissible time steps and dampens the performance of the fixed point iteration used within the DAE solver. The regularised potential \(F_\omega\), being of class \(C^3\), removes this source of stiffness. As a result, larger adaptive time steps are accepted and fewer nonlinear iterations are required to reach the stationary state.

\begin{table}[h!]
\setlength{\tabcolsep}{0.4em}
\centering
\fontsize{8}{9}\selectfont
\def\arraystretch{1.5}
\begin{tabular}{r c c c c c c c}
\toprule
\multirow{2}{*}{ \bf Scaling $\eta$ } 
&&  \multicolumn{5}{c}{\bf Error \( \big\| \,  \rhob_\omega (10-3\sigma) - \rhob( 10)  \big\|_{L^2(\Omega)} \) }
\\  \cline{3-8}
\\[-1.1em]
&& \(\rho_0 = \) \cref{eq:NLCH_Initial_Cond_SinCos}  && \(\rho_0 \equiv -\sfrac{1}{2} \)  && \( \rho_0 = \) \cref{eq:NLCH_Initial_Cond_Conv}
%
%
%
\\
\midrule
	\hphantom{$-$}100	&& 	\num{2.39e-10}	&&	\num{2.12e-14}	&&	\num{3.79e-14}
	\\
	\hphantom{$-$0}50	&& 	\num{1.14e-10}	&&	\num{4.88e-13}	&&	\num{2.18e-13}
	\\
	\hphantom{$-$00}1	&& 	\num{3.28e-10}	&&	\num{5.62e-15}	&&	\num{7.65e-13}
	\\
	\hphantom{00}$-1$	&& 	\num{1.80e-10}	&&	\num{4.32e-15}	&&	\num{5.03e-13}
	\\
	\hphantom{0}$-50$	&& 	\num{2.18e-10}	&&	\num{4.92e-11}	&&	\num{1.01e-10}
	\\
	$-100$	&& 	\num{5.76e-9} \hphantom{.}	&&	\num{5.39e-12}	&&	\num{2.29e-12}
	\\
\bottomrule
\end{tabular}
\\[0.5em]
\caption{Approximation error for different values of \(\eta\) and initial conditions.}
\vspace{-0.5\baselineskip}
\label{tb:D_CH_PS_Errs_Long}
\end{table}

Finally, we repeated this test for several parameter configurations from the previous numerical experiments. We considered the initial conditions used earlier, together with the additional constant initial condition \(\rho_0\equiv \sfrac{-1}{2}\), and evaluated the quality of the approximation \(\rho_\omega\) for different values of the scaling parameter \(\eta\). For these experiments, we set  \( \sigma = (\sfrac{1}{3}) \times 10^{-3} \) and integrated up to time \(T = 10 - 3\sigma\) using a fixed solver tolerance of \num{e-7}.
The resulting approximation errors are reported in \cref{tb:D_CH_PS_Errs_Long}. In all cases, the error remains small (below \num{6e-9}, which already lies under the tolerance of the solver), and in several configurations, notably for the constant initial condition, it reaches machine precision. These results indicate that the regularised potential provides an effective and reliable tool for approximating stationary states of \cref{sys:NL-CH_Particular} while significantly reducing the computational effort required for long-time integration.

\subsection{A kernel mixture}\label{sec:NLCHnum_Mixed_Kernel}

As a final experiment, let us consider a nonlocal kernel that displays both positive and negative values at similar scales. 
Numerically, the discretisation \(\mathbb{K} \) of the Newtonian kernel defined over \( \Omega \) mostly displays negative entries lying inside the interval \( [-1.2 \times 10^{-3},  2.5 \times 10^{-5}] \).
The discretisation of the convolution for the mollifier \cref{eq:Mollifier}, denoted  \( \mathbb{H}_a\), instead displays non-negative entries bounded above by 
\( 6 \times 10^{-2} \) for \( a = 10^{-1}\). 
Let us now define the kernel
\[
	\mathbb{L}_\eta = \eta \del{ \mathbb{K} +  (\sfrac{1}{40}) \, \mathbb{H}_a }.
\]
The resulting kernel displays entries with values contained inside the interval \( [-6 \times 10^{-4} , 4 \times 10^{-4}] \). Thus, we have balanced the maximum and minimum values of the kernel entries.

\begin{figure}[ht!]
\centering
\subcaptionbox{
    	Formation of sharp structures for \( \eta = 1\,000\).		\label{sfig:E_Solution_PS_e}
	\vspace{0.5em}
    }
    {
    \includegraphics[scale=1.4]{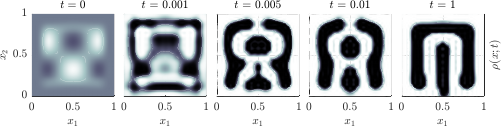}
    }
\\[0.5\baselineskip]
    \subcaptionbox{
    	Formation of sharp structures for \( \eta = 500\).		\label{sfig:E_Solution_PS_a}
	\vspace{0.5em}
    }
    {
    \includegraphics[scale=1.4]{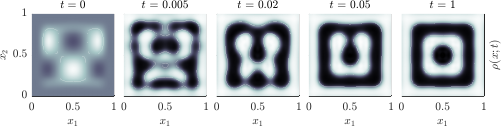}
    }
\\[0.5\baselineskip]
    \subcaptionbox{
	Formation of structures for  \( \eta = 300\).		\label{sfig:E_Solution_PS_d}
	\vspace{0.5em}
    }
    {
    \includegraphics[scale=1.4]{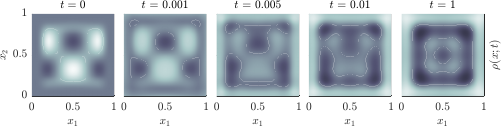}
    }
\caption{Evolution of \(\rho(x;t)\) from the initial condition \cref{eq:NLCH_Initial_Cond_Conv} and \( \eta \in  \{-5, -1, 2, 3, 5, 10\} \times 10^2 \).}
\end{figure}
\begin{figure}[ht!]\ContinuedFloat
\centering
    \subcaptionbox{
    	Concentration of values at corners for \( \eta = 200\).		\label{sfig:E_Solution_PS_c}
	\vspace{0.5em}
    }
    {
    \includegraphics[scale=1.4]{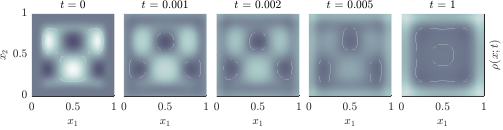}
    }
    \\[0.5\baselineskip]
    \subcaptionbox{
	Slow phase separation for \( \eta = -100\).		\label{sfig:E_Solution_PS_b}
	\vspace{0.5em}
    }
    {
    \includegraphics[scale=1.4]{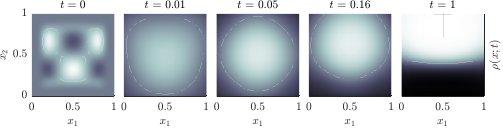}
    }
    \\[0.5\baselineskip]
    \subcaptionbox{
    	Fast phase separation for \( \eta = -500\).		\label{sfig:E_Solution_PS_f}
	\vspace{0.25em}
    }
    {
    \includegraphics[scale=1.4]{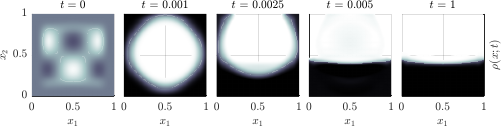}
    }
\caption{(Continued) Evolution of \(\rho(x;t)\) from the initial condition \cref{eq:NLCH_Initial_Cond_Conv} and \( \eta \in  \{-5, -1, 2, 3, 5, 10\} \times 10^2 \).}
\label{fig:E_Solution_PS}
\vspace{-0.5\baselineskip}
\end{figure}

We tested the effect of \( \mathbb{L}_\eta\) for several values of \(\eta\) inside the interval \( [-500,1\,000]\), the initial condition of compact support \cref{eq:NLCH_Initial_Cond_Conv}, and maximum time \( T = 1\).
The resulting dynamics displayed many familiar shapes that we observed in previous sections, but additional structures were formed for some values of \(\eta\). 
We collected a relevant set of examples in \cref{fig:E_Solution_PS}. 

In general, the dynamics for \( \eta < 0 \) behave similarly to what we observed for panel \embh{(b)} in \cref{fig:C_Solution_PS}. 
The only difference is that for \( \eta = -500\) the solution undergoes two separation phases, as in this case the solution quickly approaches the function \( p_1(x) = 2 \del{\mathsf{1}_{B_2( \sfrac{1}{2} \Ones[2], \sfrac{1}{2} )} (x) } - 1 \) to then refine to the step function \( p_2(x) = 2 \del{\mathsf{1}_{ \R_{\geq 0} } (x_2 - \sfrac{2}{5} )} - 1 \).

The dynamics for \( \eta > 0\) exhibit significant changes to what we observed in panel \embh{(a)} of \cref{fig:C_Solution_PS}. In that setting, we observed that the Newtonian potential did not promote the formation of additional structures besides the concentration of density at the corners. In this new setting, we observe instead that \( \mathbb{L}_\eta\) promotes the formation of additional energy-minimising structures. The most significant cases are the ones displayed for \( \eta \in \{500, 1\,000 \}\) where the phase separation property forms distinguishable and pattern-rich structures.

\section{Conclusions}\label{ch:NLCH_Conclusions}

We have successfully found solutions to the nonlocal Cahn--Hilliard system with constant mobility, featuring logarithmic and Newtonian potentials. Our study began with a review of the system, presenting an extension of the existence and uniqueness theory to polygonal domains. The Newtonian potential displayed a weak singularity at the origin, limiting the applicability of direct collocation methods. To overcome this limitation, we have developed a spectral element method to approximate convolution using a multishape approach and a piecewise-constant approximation. Analytic bounds on the error and extensive numerical tests allowed us to obtain error-controlled approximations that can be efficiently employed to solve differential systems. We then presented an extensive gallery of examples exhibiting the effectiveness of the method and the formation of rich structures that display the phase-separation property. For the first time, we have found high-resolution numerical solutions for this system that can be achieved with limited computational resources. Our future work will be centred on solving coupled systems of partial differential equations that 
include the nonlocal Cahn--Hilliard system to study control problems, e.g., the Hawkins--Daruud model for tumour growth or the Cahn--Hilliard--Navier--Stokes system, to recover biological data or mediate chemical reactions.

\subsubsection*{Funding}

A.M-T. acknowledges support from the MAC–MIGS CDT Scholarship under the Engineering and Physical Sciences (EPSRC) UK grant EP/S023291/1, and the Postdoctoral Pathways for Growth in the UK scheme, administered by the University of Edinburgh.
J.W.P. acknowledges support from the EPSRC grants EP/S027785/1 and EP/Z533786/1. The research of AP was funded in part by the Austrian Science Fund \href{https://doi.org/10.55776/ESP552}{10.55776/ESP552}.

\subsubsection*{Acknowledgments}

A.M-T. acknowledges that part of this work was done partially as a visiting researcher during the programme Multiscale Analysis and Methods for Quantum and Kinetic Problems at the Institute for Mathematical Sciences, National University of Singapore, in 2023.
MG and AP are members of Gruppo Nazionale per l'Ana\-li\-si Matematica, la Probabilit\`{a} e le loro Applicazioni (GNAMPA), Istituto Nazionale di Alta Matematica (INdAM). MG's research is part of the activities of ``Dipartimento di Eccellenza 2023--2027'' of Politecnico di Milano. 

\newpage
\addcontentsline{toc}{part}{Bibliography}
\setlength{\bibsep}{0pt plus 0.3ex}
\footnotesize
\bibliographystyle{siamplain}
\bibliography{manuscript.bbl}


\newpage
\begin{appendices}
\renewcommand{\theequation}{A--\arabic{section}.\arabic{equation}}
\setcounter{equation}{0}

\section{Diffeomorphic transformations}\label{app:diffs}

\subsection*{Unit square to general rectangle}

Any open rectangle \(\Theta = (a_1,b_1) \times (a_2,b_2) \) can be converted to \(\Omega\) through the inverse of the following affine transformation:
\begin{wrapfigure}[7]{r}{0.3\textwidth} 
    \centering
    \vspace{-1.2em}
    \includegraphics[scale=0.5,page=1]{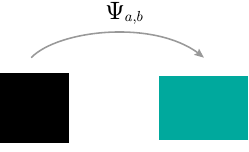}
    \caption{Square to rectangle map.}
    \label{sub:Map_Rectangle}
\end{wrapfigure}
\[
	\Psi: \overline{\Omega} \ni (x_1,x_2)  \longmapsto \del[1]{  [b_1 - a_1] x_1 + a_1, [b_2 - a_2] x_2 + a_2 } 
	\in \overline{\Theta},
\]
which is given by
%
%
\[
	\Psi^{-1}: \overline{\Theta}  \ni (y_1,y_2)  \longmapsto \del{ \frac{ y_1 - a_1 }{b_1 - a_1} , \frac{ y_2 - a_2 }{b_2 - a_2}  } \in \overline{\Omega}.
\]
For convolutions, this transformation yields 
\(
	\abs[1]{ \det\del[0]{J_\Psi } }(x) =  |b_1-a_1| |b_2 - a_2| 
\) for any \(x\in \overline{\Omega}\).


\subsection*{Unit square to bulged square}

In biological applications, individual cells often feature curved membranes and non-orthogonal structures that rectangular domains do not represent well. To illustrate how our framework extends to such geometries, we deform the unit square into a convex, ``bulged'' domain, which, although not used elsewhere in the paper, serves as a geometric template for future research on phase separation in irregular cellular domains.

A convex and open bulged square \(\Theta\) is a square with bulging (outer) sides and pointy corners. The bulge can be controlled by a parameter \(k\in \R\), where \( k \in [0,\sfrac{1}{2})\) returns a convex but cornered set, \(k \geq \sfrac{1}{2}\) gives a rounded set but \(\Theta\) is no longer one-to-one, \(k \in (-\sfrac{1}{2}, 0) \) gives a pillow square shape, and \(k \leq -\sfrac{1}{2}\) creates self-intersecting shapes. We are interested in the first case, with the map:
\begin{wrapfigure}[7]{r}{0.25\textwidth} 
    \centering
    \vspace{-1.2em}
    \includegraphics[scale=0.5,page=2]{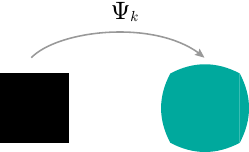}
    \caption{Bulged square map.} 
    \label{sub:Map_Bulge}
\end{wrapfigure}
\[
    \Psi: \overline{\Omega} \ni (x_1,x_2)  
    \longmapsto \del{  (2x_1-1) \left[ 1 + 4 k x_2 \del{1- x_2} \right], (2x_2-1) \left[ 1 + 4k x_1 \del{1- x_1} \right] } 
	\in \overline{\Theta},
\]
Here
\begin{align*}
	\det \del[0]{ J_\Psi } (x) 
	&= 
	\begin{vmatrix}
		2 + 8 k x_2 \del{1- x_2} & -4k (1-2x_1) (1-2x_2)
		\\
		-4k (1-2x_1) (1-2x_2) & 2 + 8 k x_1 \del{1- x_1}
	\end{vmatrix}
	\geq 4(1 - 4 k^2),
\end{align*}
%
and the latter is positive whenever \( k \in (- \sfrac 1 2,\sfrac 1 2) \)
Notice that from \cref{eq:Conv_Change_of_Vars}, there is no need to explicitly have \(\Psi^{-1}\) at hand.


\section{Minimal partition}\label{App:MinPart_MS}

\begin{figure}[h!]
\begin{center}
	\includegraphics[scale=0.6, page=1]{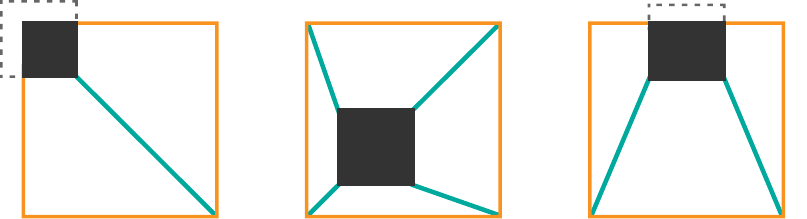}
	\caption{Minimal quadrilateral partitions based on 
	a neighbourhood around a collocation point.}
\label{fig:Multishape_Minimal}
\end{center}
\vspace{-0.5\baselineskip}
\end{figure}

We can consider a partition of \(\dsquare\) based on the minimal number of quadrilaterals required to determine a multishape. There are three scenarios to consider:  
\textbf{(a)}
\( \scaleobj{0.75}{\blacksquare} \) is located at a corner of \(\overline{\square}\). Then, we only require \embh{2} elements to partition \(\dsquare\). These elements are determined by tracing a line from the corner of \( \scaleobj{0.75}{\blacksquare} \) not in the border of \(\square\) towards the opposing corner of \(\square\).
\textbf{(b)}
\( \scaleobj{0.75}{\blacksquare} \) is completely embedded inside \(\square\) (i.e., \( \scaleobj{0.75}{\blacksquare} \Subset \square \)). Then we require \embh{4} elements determined by tracing a line from each corner of \( \scaleobj{0.75}{\blacksquare} \) to its correspondent in corner in \(\square\) (i.e., connect each pair of bottom-left, bottom-right, top-left, and top-right corners).
%
\textbf{(c)}
\( \scaleobj{0.75}{\blacksquare} \) shares one side, and one side only, with \(\overline{\square}\). Then we need \embh{3} elements, which can be determined similarly as in case \textbf{(b)}.

The procedure is depicted in \cref{fig:Multishape_Minimal}. There, a black-solid square represents \( \scaleobj{0.75}{\blacksquare} \) while the transversal segments inside \(\square\) (in teal) represent the additional lines required to determine a multishape in \(\dsquare\). 
Observe that if \( (\alpha N)^2\) is the number of collocation points used to discretise \( \scaleobj{0.75}{\blacksquare} \), then we will deal with at most \( 4 (\alpha N)^2 \) points at once for most iterations in the inner loop of \cref{alg:MultiShape}. This latter supposes an improvement in contrast to the \( 8 (\alpha N)^2 \) points needed for the maximal partition in \cref{ssec:Part_MS}. 

However, a validation step can showcase the limited accuracy of the approximation. To see this, we can run a similar numerical experiment as in \cref{ssec:Quality_test}. Under this partition, we computed the numerical error \cref{eq:MS_error_NP_conv_abs} for \( \alpha \) in the set \( \{4,5,8,10\} \), letting
 \( \varepsilon \in \{10^{-2}, 10^{-5}\}\) and \(N \in \{10,20\}\).

The results are presented in \cref{fig:K_Test_MinQ_NP_alpha_10}. As in \cref{ssec:Quality_test}, we observe that the refinement strategy improves the convergence of the method. In particular, we observe that the distribution of the error is concentrated around the median and mode values of each swarm. Furthermore, we observe that spectral convergence is attained for \(N=20\) and \(\varepsilon = 10^{-2}\). However, for \(\varepsilon = 10^{-5}\), the improvement, although noticeable, is not as drastic for any \(N\). This difference is not solely attributable to the size of \(N\). Observe that employing the same scaling factor for the two meshes implies that \(N=20\) utilises twice as many points for each experiment. This distinction is more pronounced in the case \(\varepsilon = 10^{-2}\) for the pairs of swarms \(10/8\) and \( 20/4\) or \( 10/10\) and \(20/5\), as these are essentially identical. We do not observe this behaviour for the other choice of  \(\varepsilon\). Furthermore, we observe that, in this case, for \(N=10\), the error does not go below \( 10^{-10}\), and for \( N = 20\) the bound decreases to \( 10^{-11}\). This raises the question of whether there might be some additional factor that prevents the error from falling below these observed thresholds.

\begin{figure}[ht!]
\centering
\vspace{1.75cm}
\setlength{\unitlength}{1cm}  
    \subcaptionbox{Swarm plots for scalings of \(N=10\).}[7.9cm]{
    \begin{picture}(8,5.2)  
	\put(0,0){
	    \includegraphics[scale=0.98]{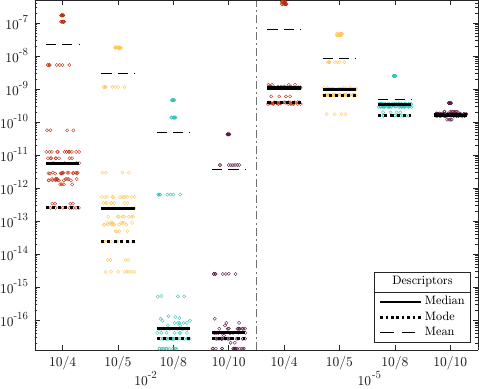}
	}
	\put(4.2,-0.3){\makebox(0,0)[c]{\fontsize{6.5}{7.5}\selectfont Instance configuration}}
	\put(-0.35,3){\rotatebox{90}{ \fontsize{6.5}{7.5}\selectfont Error $e_{\varepsilon}$} }
	\end{picture}
	\vspace{0.25em}
	}
	\hspace{2em}
    \subcaptionbox{Swarm plots for scalings of \(N=20\).}[7.9cm]{
    \begin{picture}(8,5.2)  
	\put(0,0){
	    \includegraphics[scale=0.98]{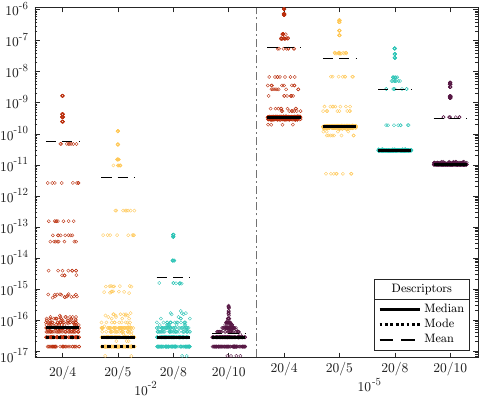}
	}
	\put(4.2,-0.3){\makebox(0,0)[c]{\fontsize{6.5}{7.5}\selectfont Instance configuration}}
	\put(-0.35,3.){\rotatebox{90}{ \fontsize{6.5}{7.5}\selectfont Error $e_{\varepsilon}$} }
	\end{picture}
	\vspace{0.25em}
	}
	%
	%
    \caption{
	Swarm plots of the relative errors \(e_{\varepsilon}\) for approximating the Newtonian potential under refinement. The horizontal labels are presented in the format \( N/\alpha\), corresponding to the grid resolution 
	$N \in \{10,20\}$ and the scaling factor $\alpha \in \{4,5,8,10\}$. Additionally, two options for the neighbourhood radius are presented, specifically $\varepsilon \in \{ \num{e-2}, \num{e-5} \}$.
    }
\label{fig:K_Test_MinQ_NP_alpha_10}
\end{figure}


The geometric construction in \cref{fig:Multishape_Minimal} presents an issue as \(\varepsilon\) becomes small: Some computational quadrilaterals start to degenerate into other shapes. Specifically, there are some quadrilaterals that behave like triangles and others behave like parallel lines to the boundary of the box. 
In principle, this is not a problem as long as the centre of the small box is not close to the boundary. However, if a collocation point is close to the boundary (e.g., minimum coordinate distance in \((2\varepsilon,3\varepsilon)\)), then the inner angles of at least one element in the partition can grow very close to \(\pi\) [rad]. This can generate a problem with the numerical structures used to compute quantities in the multishape. 
%
%
Indeed, a close examination of the pointwise errors reveals that the error displays higher values close to the boundary of the computational domain. 
%
%
We have seen that we can slightly overcome this issue by increasing the scaling factor \(\alpha\). However, even in such regimes, the error will still be concentrated close to the boundary where very sharp or very flat angles form inside the elements.

\section{Solution on the disc}\label{App:Disc}

We can extend our workings from \cref{ch:CH_WP_and_Setup} to the disc domain \(\Omega \coloneqq B_2(0;1) \), observing that the theory presented in \cref{ch:Spectral_Element_Newtonian_Potential} also holds. First, let us introduce the graphical notation
\( \Omega \eqqcolon \disc\), \( \Omega \setminus B_2 (x; \varepsilon)  \eqqcolon \ddisc \), and \( \Omega \cap B_2 (x; \varepsilon)  \eqqcolon \raisebox{.05em}{\smalldisc} \).

To apply our methods, we need two ingredients: a computational correction factor like \cref{eq:G_Diagonal_Approximation_Newtonian} and a partition strategy as in \cref{ssec:Part_MS}. The former aims to approximate 
\begin{equation}\label{eq:G_Disc}
	\mathtt{G}_\varepsilon(x) 
	\coloneqq
	\int\limits_{x - \smalldisc} K(u) \dif u.
\end{equation}
This is a particularly difficult quantity to obtain analytically for \( \operatorname{dist}(x,\partial\disc) < \varepsilon\), where a circular segment needs to be considered. Alternatively, we can approximate \( \mathtt{G}_\varepsilon \) by the relation
\begin{equation}
\label{eq:J_Disc}
	\mathtt{J}(x) \coloneqq 
	K \star 1_{\disc} (x) = \mathtt{G}_\varepsilon(x) + \int\limits_{\ddisc}  K(x-y) \dif y,
	\qquad \forall x \in \overline{\disc}.
\end{equation}
Thus, if we can numerically establish the value of the integral over \( \ddisc\), then we can approximate \( \mathtt{G}_\varepsilon \) from \cref{eq:J_Disc} and use this correction factor within the numerics. Notice that some additional numerical error could be introduced from the quadrature method. Notwithstanding, since integration in multishape is based on the Clenshaw--Curtis quadrature weights and \(K\) is smooth on the latter domain, we can assert that the approximation will have spectral accuracy \cite{Trefethen2008}. We tested this assumption in our numerical tests, suggesting a path for future work in cases where an exact \(\mathtt{G}_\varepsilon\) may not be available. 

It is possible to determine \cref{eq:J_Disc} analytically by realising that \(\mathtt{J}\) is radially symmetric since \(K\) is so as well. Then, we can reduce the convolution problem to \( x\in [-1,1] \times \{0\}\) and consider a change of variables to polar coordinates; this process involves the computation of log-cosine integrals \cite[\S 4.224]{Zwillinger2014a}. 
Alternatively, from potential theory, we have that \( J\) is the solution of \( \Delta J = 1\) in \(\disc\) with homogeneous Dirichlet boundary condition. Combining radial symmetry with a change of polar coordinates reveals that this is just an Euler--Cauchy equation yielding the solution
\begin{equation*}
	\mathtt{J}(x) = \frac{1}{4} \del{ \|x\|^2 -1 }.
\end{equation*}

An exact formula for \( \mathtt{G}_\varepsilon \) can be found by analysing the geometry of the intersection between the two discs that define {\raisebox{.05em}{\smalldisc}}. Thus, the convolution of the Newtonian potential over {\raisebox{.05em}{\smalldisc}} results in a highly nontrivial expression involving the dilogarithm \( \operatorname{Li}_2 :\C\to \C\). We refer to the following result from \cite{AMT2025a}, which also provides details on the efficient evaluation of \(\mathtt{G}_\varepsilon\) as \(\varepsilon\) becomes small:

\begin{lemma}[{\citealp[Theorem 1]{AMT2025a}}]\label{th:Conv_NonTrivial}
	Let \( \mathtt{G}_\varepsilon  \coloneqq K \star \Ones[ \raisebox{.05em}{\smalldisc} ] \), then for any \( x \in \disc\), we have that
	\begin{equation}
	\label{eq:Disc_Exact_G}
		\mathtt{G}_\varepsilon (x) = \frac{1}{2\pi} \int\limits_{ x - \smalldisc } \log |y| \dif y
		=
		\frac{1}{4}
		\begin{cases}
			\varepsilon^2 (\log \varepsilon^2 - 1) & \text{if } x \in B[0;1-\varepsilon],
			\\
			\dfrac{1}{\pi} (\pi - \varphi) \varepsilon^2 ( \log \varepsilon^2 - 1) + H\big(\|x\|,\varphi\big)  
			& \text{if } 
			\|x\| \in (1-\varepsilon, 1],
		\end{cases}
	\end{equation}
	where \(\theta \coloneqq \arccos \left( \frac{1 - \|x\|^2 - \varepsilon^2}{2 \|x\|\varepsilon} \right) \) and \( H: [1-\varepsilon, 1] \times [0,2\pi] \to \R\) is given by
	\begin{align*}
		H(a,\theta) &= \left. \frac{2}{\pi} \left[
		 \Im \big( \operatorname{Li}_2 (-a e^{2i\phi}) \big)
		 + (1-a^2) \left[ \phi - \frac{1}{2}  \arctan\left(\frac{ a \sin 2\phi}{1 + a \cos 2\phi} \right)  \right]
		+ a ( 1- \log \varepsilon ) \sin 2\phi \right] 
		\right|_{\phi = \Phi(\theta) } - (1-a^2),
	\end{align*}
	with \( \Phi(\theta) = \frac{1}{2} \arccos L(\varphi) \) and 
	\(
	L(\theta) = \frac{ \varepsilon^2  - 1 - a^2}{ 2 a }
	\).
\end{lemma}


Notice that the error estimate from \cref{th:QualityApprox} still holds as two main properties from \cref{lem:G_analytical_expression} also apply to \cref{eq:G_Disc}; i.e., \( \mathtt{G}_\varepsilon(x) \) is strictly negative and symmetric. In particular, the monotony of \(\log |r|\) allows us to conclude that the disc \( B_2(0;1-\varepsilon)\) forms the set of global minima of \( \mathtt{G}_\varepsilon\). In particular, the first branch of \cref{eq:Disc_Exact_G} provides the relation
\( \mathtt{G}_\varepsilon (x) \sim O( \varepsilon^2 \log \varepsilon )
\).
Since \cref{eq:CH_Conv_Neighbourhood_eps_min} is a lower bound of \cref{eq:G_Disc}, it thus yields an upper bound on the approximation error of \( K\star \rho\).

\subsection{Partition method}

\begin{figure}[!ht]
\begin{center}
	\includegraphics[scale=0.6, page=1]{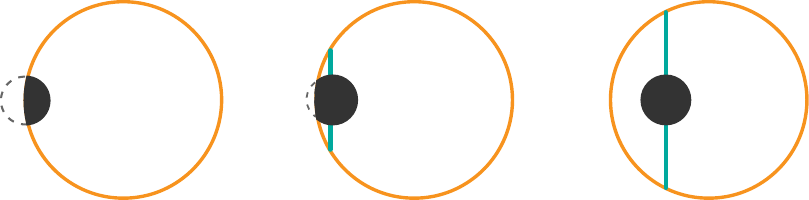}
	\caption{Proposed partition method for a disc based on a neighbourhood around a collocation point.}
\label{fig:MS_Disc}
\end{center}
\vspace{-0.5\baselineskip}
\end{figure}

Considering rotational symmetry, we can centre \(B_2(x;\varepsilon)\) at the horizontal axis, maybe after a rotation. We propose a partition scheme where we trace a vertical line passing through the poles of the \(\varepsilon\) ball, whenever these are included in the intersection. Thus, any partition of \(\ddisc\) can be cast as one of the three cases depicted in \cref{fig:MS_Disc}. From left to right, the first case occurs whenever  \( \operatorname{dist}(x,\partial\disc) = 0\). Here we obtain a shape that resembles a crescent moon, which we call a \embh{moon gap}. The second case occurs whenever \( \operatorname{dist}(x,\partial\disc) \in (0,\varepsilon]\), where we obtain three shapes: one moon gap with two straight edges and two triangular-resembling bits, which we call \embh{lune shards}. Finally, when \raisebox{.05em}{\smalldisc} is compactly contained inside \(\disc\), i.e., \( \operatorname{dist}(x,\partial\disc) \in (\varepsilon,1]\), we obtain a partition involving two moon gaps with straight edges. 


\begin{figure}[h!]
\centering
\captionsetup[subfloat]{format=hang,singlelinecheck=false}
    \subcaptionbox{\texttt{MoonGap} with no edges.}{
    \includegraphics[scale=0.6, page=2]{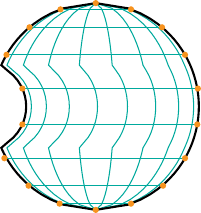}
	\vspace{0.25em}
	}
	\hspace{1.25cm}
    \subcaptionbox{\texttt{MoonGap} with two straight edges.}[2.25cm]{
    \includegraphics[scale=0.6, page=3]{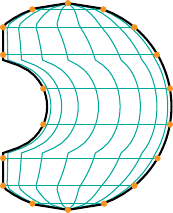}
	\vspace{0.25em}
	}
	\hspace{1.25cm}
    \subcaptionbox{\texttt{LuneShard}.}[1.6cm]{
    \includegraphics[scale=0.6, page=4]{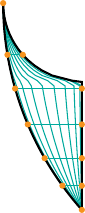}
	\vspace{0.25em}
	}
	%
	%
    \caption{New shapes added to MultiShape for disc partition.}
\label{fig:New_Shapes}
\end{figure}

To use this partition in conjunction with MultiShape, we created two new shapes: \texttt{MoonGap} and \texttt{LuneShard}. Both depend on the distance between the centres of the two discs \( d\) and a sense parameter \(s\) determining the direction (left/right or up/down) of the shape with respect to the vertical line that divides \(\disc\), if present.

\begin{remark}
	Observe that tracing two parallel lines from the intersection points between discs yields a partition of \(\disc\) consisting of two circular segments and at least one additional set that resembles a portion of the moon gap or the lune shard. We tested this additional partition scheme and obtained similar results to those of \cref{App:MinPart_MS}.
\end{remark}

\begin{figure}[!ht]
\centering
	\includegraphics[scale=0.98]{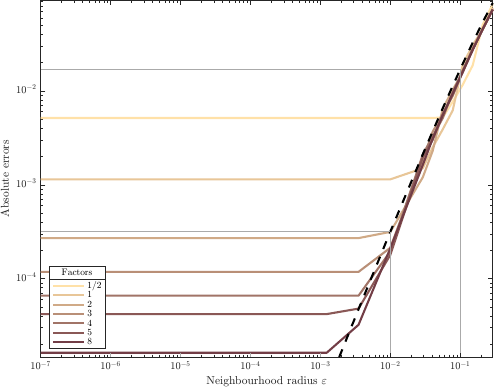}
	%
	%
    \caption{
	Evolution (in log-log scale) of the absolute error (vertical axes) as \(\varepsilon\) varies in \( [10^{-7}, \sfrac{3}{10}]\) (horizontal axes). The dashed line is given by the maximum of \( |\mathtt{G}_\varepsilon|\); see \cref{eq:Disc_Exact_G}.
    }
\label{fig:Disc_Error}
\vspace{-1\baselineskip}
\end{figure}


Replicating some of our tests from \cref{ssec:Quality_test}, we present in \cref{fig:Disc_Error} an analogue of panel (b) in \cref{fig:Multishape_Maximal_Q_MaxError}. Here, we computed the absolute error of computing \( \mathtt{J}(x) \) against the numerical approximation of 
the integral of \( K(x-\, \cdot) \) over \({\ddisc}\)
using the proposed partition method. We computed the absolute difference between these quantities as \(\varepsilon\) decreases for seven scaling factors \( \alpha \in \{ \sfrac 1 2, 1, 2,3,4,5,8 \}\).
Once more, we see that the approximated integral inside the multishape 
quickly approaches \(|\mathtt{G}_\varepsilon|\) for moderate values of \(\varepsilon\). Increasing the scaling factor reduces the correction term's impact for smaller values. Thus, we can select a suitable pair \( (\varepsilon,\alpha)\) to significatively reduce the error term while also reducing the contribution of \( \raisebox{.05em}{\smalldisc} \) as much as possible.

\subsection{Numerical tests}

We assembled approximate convolution kernels based on the following parameter configurations:
\(
	(N, \alpha) \in 
	\big\{ (20, 4), (40, 2), (60, \sfrac{3}{2}) \big\},
\)
and fixing \(\varepsilon = \num{e-3}\).
Using the resulting kernels, we computed solutions of \cref{sys:NL-CH_Particular}, employing the scaled kernel \(K_{\eta}\) defined in \cref{eq:Nonlocal_Energy_NP}, with
\(
\eta \in \{100, 50, -10, -20\}
\). We also set \(T=100\) to capture any long-time behaviours.

The simulations were carried out for the following two initial conditions:
\begin{equation}
\label{eq:App_C_Discs_Initial}
\begin{aligned}
	\rho_{0}^{(1)} (x) &= \frac{1}{2} G\del{ \frac{x_1}{2}, \frac{x_2}{2}, 0 } - \frac{1}{10},
	\qquad\text{where}\qquad
	G(x_1,x_2;a) \coloneqq \exp \del{ -\frac{1}{0.05^2} (x_1^2 + x_2^2 - a)^2},
	\\
	\rho_{0}^{(2)} (x) &= 
	\sqrt{ 1 - x_1^2 - x_2^2 } + \frac{1}{2} \del[1]{ G(x_1,x_2;1) + G(x_1,x_2;-1) - 1 }
	.
\end{aligned}
\end{equation}
The first condition, \(\rho_{0}^{(1)}\), labelled \embh{sombrero}, satisfies the prescribed boundary data by construction. In contrast, the second condition, \(\rho_{0}^{(2)}\), labelled \embh{bowler}, provides continuous initial data with reduced regularity at the boundary.

\begin{figure}[ht]
\centering
\subcaptionbox{
    	Initial condition and final values corresponding to $\rho(x;0) = \rho_0^{(1)}(x)$.
	\vspace{0.5em}
    }
    {
    \includegraphics[scale=1.4, page = 2]{Graphics/Disc/Hats}
    }
\\[0.5\baselineskip]
    \subcaptionbox{
	Initial condition and final values corresponding to $\rho(x;0) = \rho_0^{(2)}(x)$.
	\vspace{0.5em}
    }
    {
    \includegraphics[scale=1.4, page = 1]{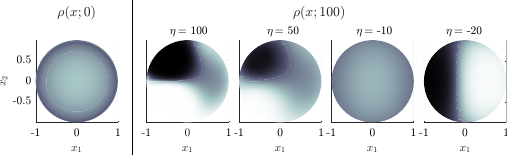}
    }
\caption{Final states of \(\rho(x;t)\) starting from the initial conditions \cref{eq:App_C_Discs_Initial} with \( \eta \in  \{10, 5, -1, -2\} \times 10 \).}
\label{fig:App_C_Sols_60}
\vspace{-0.5\baselineskip}
\end{figure}

Similar to the case of the box, the obtained solutions were consistent across mesh refinements and underwent a regularisation and specialisation phase that minimised the nonlocal energy associated with the kernel $K_\eta$. In \cref{fig:App_C_Sols_60}, the terminal states \( \rho(T=100)\) are depicted for the \(N=60\) multishapes. Following the same colour convention as in the box, values close to \(-1\) correspond to near-black shades, while values close to 1 appear as the brightest areas. Intermediate grayscale tones represent values around zero. In all cases but \( \eta = -10\), we observe the phase separation property that either resembles a step function over a quadrant or a half plane.

\section{Solution on the cube}\label{App:Cube}

The pseudospectral \texttt{2DChebClass} package has a three-dimensional extension, developed in the same spirit as its two-dimensional counterpart. It provides Chebyshev--Gauss--Lobatto collocation grids together with the associated differentiation, interpolation, and integration operators on the cube. Since the multishape framework has not yet been extended to the three-dimensional class, all computations in this appendix are performed directly on the cube without geometric splitting. This choice allows us to obtain high-quality solutions to the nonlocal system \cref{PDS:NL-CH} without an explicit singular kernel.


The three-dimensional Newtonian potential is given by
\begin{equation}\label{eq:Newtonian_potential_3D}
	K: \R^3 \setminus\{0\} \ni x  \longmapsto  -\frac{1}{4\pi \|x\|}  \in \R_{< 0}.
\end{equation}
The kernel \(K\) is radially symmetric and differentiable, with a weak and integrable singularity at the origin. The admissibility conditions (see \cref{Re:AlternativeAdmissibility}) follow directly from the three-dimensional analogue of \cref{lem:Suitability_of_Kernel}.

To approximate the quadrature values of the convolution, we introduce the regularised kernel
\[
	K_\sigma : \R^3 \setminus\{0\} \ni x  \longmapsto 
	-\frac{1}{4\pi \max\{ \sigma, \|x\| \} }
	\in \R_{< 0}.
\]
A simple choice for \(\sigma\) is to select a value smaller than the minimum distance between any two collocation points. Here we recall that, for a line of \(N\) Chebyshev--Gauss--Lobatto points, the smallest distance between two pseudospectral points corresponds to \( \Delta x = 1 - \cos ( \sfrac{\pi}{(N-1)} ) \), and this is replicated in the Chebyshev cube. 

The next result shows us the quality of approximating the convolution action of \cref{eq:Newtonian_potential_3D} with $K_\sigma$:
\begin{theorem}\label{th:QualityApprox_Cube}
	Let \(\rho \in {L^\infty(\text{\faCube})} \) be bounded inside \([0, \nu]\) for some \(\nu > 0\). 
	Also, let \( \sigma = \kappa \Delta x\), with \( \kappa \in (0,1)\), and \( p \in [1,\infty]\). Then
	\begin{align}
		\norm{ K\ast \rho - K_\sigma \ast \rho }_{L^p (\text{\faCube})}
		\leq \frac{1}{6} \nu \sigma^2 8^{\sfrac 1 p}
		\leq
		\frac{\nu \kappa^2}{24} \frac{\pi^4}{ (N-1)^4 } 8^{\sfrac 1 p} \sim O(N^{-4}).
		\label{eq:Approx_Error_Estimate}
	\end{align}
\end{theorem}
\begin{proof}
	The result follows from a standard truncation argument. Observe that, by definition, \(K - K_\sigma\) is zero except on the Euclidean neighbourhood \( \|u\| \leq \sigma\). There, a transformation to spherical coordinates yields
	\[
		\int\limits_{\R^3} \abs{ K(u) - K_\sigma (u) } \dif u = \frac{1}{4\pi} \int\limits_{B_2(0;\sigma)} \frac{1}{\|u\|} - \frac{1}{\sigma} \dif u
		=
		\int\limits_{0}^\sigma r - \frac{r^2}{\sigma} \dif r = \frac{\sigma^2}{6}.
	\]
	Then, for any \( x \in \text{\faCube}\), we can obtain the general bound
	\[
		\abs{ K\ast \rho (x) - K_\sigma \ast \rho (x) }
		\leq \| \rho \|_{L^\infty} \hspace{-0.5em} \int\limits_{x - \text{\faCube}} \abs{ K(u) - K_\sigma (u) } \dif u
		\leq \nu \int\limits_{\R^3} \abs{ K(u) - K_\sigma (u) } \dif u = \frac{\nu \sigma^2}{6}.
	\]
	The final approximation error follows from the upper bound \( \Delta x \leq \sfrac{\pi^2}{2(N-1)^2} \).
\end{proof}

The value of \cref{eq:Approx_Error_Estimate} decreases as \(N\) grows. However, each pseudospectral matrix (e.g., differentiation, interpolation, or convolution) is of size \( (3N)^3\). Consequently, only relatively small values of \(N \sim 20\) are computationally feasible. The rôle of \(\sigma\) is two-fold. On the one hand, taking a small \(\sigma\) guarantees that the discretised convolution matrix becomes a column-diagonally dominant matrix. On the other hand, selecting \(\sigma\) to be an arbitrarily small value significantly dampens the effects of the other entries of the matrix by spectrally reducing the action of the kernel to the product of the regularised diagonal values and the quadrature weights. An immediate effect is the loss of long-range effects, which are at the core of employing the nonlocal terms. Thus, selecting \(\sigma\) to be of the same order of magnitude as \(\Delta x\) serves as a practical heuristic for preserving long-range effects while simultaneously reducing the approximation error.

\begin{figure}[ht]
\centering
    {
    \includegraphics[scale=1.4, page = 1]{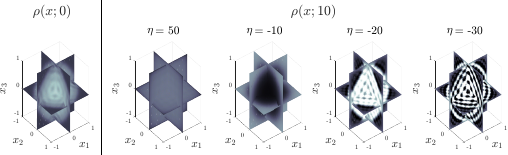}
    }
\caption{Final states \(\rho(x;T)\) obtained from the adapted initial condition \cref{eq:App_C_Discs_Initial}, for \( \eta \in \{5, -1, -2, -3\} \times 10 \).}
\label{fig:App_D_Sols_20}
\vspace{-0.5\baselineskip}
\end{figure}

We performed a numerical test using the scaled kernel \( K_\eta \coloneqq \eta K_\sigma\). We fixed \( \sigma = (\sfrac{1}{2}) \Delta x\) with \(N = 20\) and considered the values \( \eta \in \{50,-10,-20,-30\}\). As an initial condition, we extended the two-dimensional bowler \( \rho_0^{(2)} \) from \cref{eq:App_C_Discs_Initial} to employ a three-dimensional Gaussian. To capture the long-time behaviour of the system, we set \(T=10\).
The results of the numerical integration are presented in \cref{fig:App_D_Sols_20}, where we compare the initial condition with the final state of the solution \( \rho(x;T)\). Each panel in the plot shows cross-sectional intensity slices of \(\rho\) taken through the origin. Following the colour convention of the previous sections, values close to \(-1\) correspond to near-black shades, while values close to \(1\) appear as the brightest areas. Intermediate grayscale tones represent values around zero. We observe that for positive scalings the kernel diffuses, leading \(\rho(T)\) to become nearly constant, whereas for negative values of \(\eta\) the solution undergoes a separation and specialisation phase.



\end{appendices}

\end{document}